\documentclass[11pt,a4paper, parskip=half]{scrartcl} %option parskip=half
\usepackage[utf8]{inputenc}
\usepackage[english]{babel}
\usepackage[T1]{fontenc}
\usepackage{lmodern}
\usepackage[left=3cm,right=3cm,top=3cm,bottom=3cm]{geometry}
\usepackage[runin]{abstract}
    \abslabeldelim{.}

\usepackage{todonotes}

\usepackage{amsmath, amsfonts, amssymb, amsthm, amstext}
\usepackage{mathtools}
\usepackage{mathrsfs}
\usepackage{mathdots}
\usepackage{cases}

\usepackage{graphicx}
\usepackage{subcaption}
\usepackage{algorithm}

\usepackage{xcolor}

\font\mfett=cmmib10 at11pt
 at9pt
\def\balpha{\hbox{\mfett\char011}}

\usepackage{enumitem}

\usepackage{hyperref} %, nameref}
\usepackage[capitalise]{cleveref}
	
\newcounter{thm}
\numberwithin{thm}{section}
\numberwithin{equation}{section}

	\newtheoremstyle{myplain}		% name of the style to be used
			{}			% measure of space to leave above the theorem. E.g.: 3pt
			{}			% measure of space to leave below the theorem. E.g.: 3pt
			{\itshape}				% name of font to use in the body of the theorem
			{}				% measure of space to indent
			{\sffamily\bfseries}				% name of head font
			{.}		% punctuation between head and body
			{ }				% space after theorem head; " " = normal interword space
			{\thmname{#1}\thmnumber{ #2}\textnormal{\textsf{\thmnote{ (#3)}}}}			% Manually specify head
    \newtheoremstyle{mybreak} 
            {}{}{}{}{\sffamily\bfseries}{.}{\newline}
            {\thmname{#1}\thmnumber{ #2}\textnormal{\textsf{\thmnote{ (#3)}}}}
	\newtheoremstyle{mydef}
			{}{}{}{}{\sffamily\bfseries}{.}{ }
			{\thmname{#1}\thmnumber{ #2}}
	\newtheoremstyle{myrem}
			{}{}{}{}{\sffamily\itshape}{.}{ }
			{\thmname{#1}\thmnumber{ #2}}

\theoremstyle{myplain}
	\newtheorem{theorem}[thm]{Theorem}
	\newtheorem{lemma}[thm]{Lemma}
	
    \newtheorem{corollary}[thm]{Corollary}
\theoremstyle{mybreak}
\theoremstyle{mydef}
	
	\newtheorem{remark}[thm]{Remark}
\theoremstyle{mydef}
	\newtheorem{example}[thm]{Example}

\DeclareMathOperator{\diag}{diag}
\DeclareMathOperator{\sign}{sign}

\allowdisplaybreaks

\def\sumprime_#1^#2{
    \setbox0=\hbox{$\scriptstyle{#1}$}
    \setbox1=\hbox{$\scriptstyle{#2}$}
    \setbox2=\hbox{$\displaystyle{\sum}$}
    \setbox4=\hbox{${}^\prime\mathsurround=0pt$}
    \dimen0=.5\wd0 \advance\dimen0 by-.5\wd2
    \ifdim\dimen0>0pt
        \ifdim\dimen0>\wd4 \kern\wd4
        \else\kern\dimen0
        \ifdim\dimen1>\wd4 \kern\wd4
        \else\kern\dimen1
    \fi\fi\fi
\mathop{{\sum}^\prime}_{\kern-\wd4 #1}^{\kern-\wd4 #2}
}

\title{\Large Spline Representation and Redundancies of One-Dimensional ReLU Neural Network Models}
\author{
Gerlind Plonka\footnote{Institute for Numerical and Applied Mathematics, G\"ottingen University, Lotzestr.\ 16-18, 37083 G\"ottingen, Germany, \{plonka,y.riebe,kolomoitsev\}@math.uni-goettingen.de} \footnote{Corresponding author} \quad 
Yannick Riebe$^{*}$ 
\quad Yurii Kolomoitsev$^{*}$}
\date{}

\begin{document}
	\let\oldproofname=\proofname
	\renewcommand{\proofname}{\itshape\sffamily{\oldproofname}}

\maketitle

\begin{abstract}
We analyze the structure of a one-dimensional  deep  ReLU neural network (ReLU DNN) in comparison to the model of continuous piecewise linear (CPL) spline functions with arbitrary knots.  In particular, we give  a recursive  algorithm to  transfer the parameter set determining the ReLU DNN into the parameter set of a CPL spline function. Using this representation, we show that after removing the well-known parameter redundancies of the ReLU DNN, which are caused by the positive scaling property, all remaining parameters are independent.  
Moreover, we show that the ReLU DNN with one, two or three hidden layers can represent  CPL spline functions  with $K$ arbitrarily prescribed  knots (breakpoints), where $K$ is the number of real parameters determining the normalized ReLU DNN (up to the output layer parameters). Our findings are useful to fix a priori conditions on the ReLU DNN to achieve an output with prescribed breakpoints and function values.\\[1ex]

\end{abstract}

\section {Introduction}

In this paper, we present a detailed analysis of the structure of a ReLU (deep) neural network (ReLU DNN) model with input and output layers of dimension $n_{0}=n_{L}=1$ and with $L-1$ hidden layers of widths $n_{1}, \, n_{2}, \ldots , n_{L-1}$, where $n_{\ell} \in {\mathbb N}$.  
This model is for $t \in {\mathbb R}$  recursively determined by
\begin{align}  \nonumber
{\mathbf F}_{1}(t) &= t {\mathbf A}^{(1)} + {\mathbf b}^{(1)}, \\
\label{DNN}
{\mathbf F}_{\ell}(t) &= {\mathbf A}^{(\ell)} \sigma({\mathbf F}_{\ell-1}(t)) + t {\mathbf c}^{(\ell)} + {\mathbf b}^{(\ell)}, \qquad \ell=2, \ldots , L-1, \\
\nonumber 
{f}_{L}(t) &=  {\mathbf A}^{(L)} \sigma({\mathbf F}_{L-1}(t)) + t c^{(L)} + {b}^{(L)},
\end{align}
where ${\mathbf A}^{(\ell)} \in {\mathbb R}^{n_{\ell} \times {n_{\ell-1}}}$ for  $\ell=1, \ldots, L$, ${\mathbf b}^{{(1)}} \in {\mathbb R}^{n_{1}}$, ${\mathbf c}^{{(\ell)}}, {\mathbf b}^{(\ell)} \in {\mathbb R}^{n_{\ell}}$, $\ell=2, \ldots, L$. %and where $p_{1}$ is a linear polynomial, i.e., $p_{1}(t) = a^{(0)}t + b^{(L)}$ with $a^{(0)}, b^{(L)} \in {\mathbb R}$. 
Throughout this paper, we will use  the rectified linear unit (ReLU) (or linear truncated power) function as an activation function, which is applied to each component, i.e., 
$$\sigma({\mathbf x}):=(\max \{ x_{1},0\}, \, \max \{ x_{2},0\}, \ldots , \max \{ x_{n},0\})^{T} \quad \text{ for} \quad  {\mathbf x} =(x_{k})_{k=1}^{n} \in {\mathbb R}^{n}. $$
This ReLU network is slightly generalized compared to the usually used NN framework, since it admits at any layer to add beside a bias vector ${\mathbf b}^{(\ell)}$ also  a linear term $t {\mathbf c}^{(\ell)}$ for $\ell=2, \ldots , L$. At the first layer, we do not need that term  since ${\mathbf A}^{(1)}$ is already a column vector of length $n_{1}$. Observe that this special channel to copy the input $t$ has been also used in \cite{Daub22}, where it  has been called \textit{source channel}. If ${\mathbf c}^{(\ell)} = {\mathbf 0}$, we obtain the conventional (one-dimensional) ReLU  DNN.
We will denote the DNN function model in (\ref{DNN}) with $L-1$ hidden layers by
$$ {\mathcal Y}_{n_{1}, n_{2}, \ldots , n_{L-1}}, $$
and this model has $2n_{1}+ (n_{1}+2)n_{2} + (n_{2}+2)n_{3}+ \ldots + (n_{L-1}+2)n_{L}$ real parameters.

Within the last years, an overwhelming number of papers have shown how the ReLU DNN model (with $n_{0} \ge 1$)  can be  successfully used in many different applications, as for example in classification \cite{Lecun15}, feature extraction \cite{Chui19,Han22}, image denoising and restoration \cite{Dong19}.
This led also to a further theoretical investigation of this model regarding its expressivity and approximation properties, see e.g.\ \cite{Arora17,  Chen19, Daubechies21, Daub22, DeVore21, Hanin19, Lu17, Lu21, Petersen18, Shen20, Telgarski16, Yarotsky17, Zhou20}, and its connection to multivariate max-affine spline operators, see \cite{Balestriero18,Balestriero21}. These investigations include also other activation functions, see e.g.\ \cite{Boel19, Chui16}.
The universal function approximation property of a shallow neural network, the model with only one hidden layer, is well-known for different activation functions, see e.g. \cite{Chui94, Cybenko89, Hornik89} and a survey by A. Pinkus \cite{Pinkus99}.

For $n_{0}=1$, the DNN function model in (\ref{DNN}) is  closely related to the well-known model of continuous piecewise linear (CPL) spline functions $\Sigma_{N}$ with at most $N$ knots (breakpoints). Here we employ the notation that $s: {\mathbb R} \to {\mathbb R}$ is in $\Sigma_{N}$, if it can be represented in the form
\begin{equation}\label{sp}
s(t) =s_{N}(t) = q(t) + \sum_{k=1}^{N} \alpha_{k} \,  \sigma({t}-x_{k}), \qquad t \in {\mathbb R},
\end{equation}
with ordered knots $- \infty  <x_{1} < x_{2} < \ldots < x_{N} < \infty$, $\alpha_{k} \in {\mathbb R}$, and where $q(t) = q_{1} t + q_{0}$  is a linear polynomial. Obviously, $s \in \Sigma_{N}$ is continuous. 
The CPL  spline function model $\Sigma_{N}$ depends on $2N+ 2$ independent parameters. 

For a one-dimensional  ReLU shallow NN ($L=2$),  the model $\Sigma_{N}$ of free knot splines is actually equivalent with ${\mathcal Y}_{N}$ such that all approximation properties of the (adaptive) CPL spline model can be directly  carried over, see e.g. \cite{deBoor, DeVore}. 
However, while for $L>2$, $f_{L}$ in (\ref{DNN}) can still be shown to be  a CPL spline function, the number of breakpoints  of the ReLU NN can grow exponentially with the number of layers $L$, see e.g. \cite{Montufar14, Serra18}, while the number of parameters  is bounded by ${\mathcal O}(LW^{2})$, where $W= \max \{ n_{\ell}: \ell= 0, \ldots, L\}$ denotes the maximal width of the network. 
This observation indicates already the difference between the ReLU DNN and the CPL spline model for $L>2$. 
On the other hand, it has been shown  in \cite{Daub22} that a  ReLU DNN model as in (\ref{DNN})  with $n_{1}= n_{2} = \ldots = n_{L-1} =W$ depending on ${\mathcal O}(LW^{2})$ parameters  can represent any function $s \in \Sigma_{N}$ if $N \le C LW^{2}$, where $C$ is a constant being independent of $L, W$ and $N$. 

Another problem under theoretical investigation is the problem of over-parametrization and unique function representation by the ReLU DNN. One well-known transform causing parameter redundancy  is the positive scaling property. Since $\sigma({\mathbf D} {\mathbf x}) = {\mathbf D} \sigma({\mathbf x})$  for diagonal matrices ${\mathbf D} = \diag (d_{1}, \ldots , d_{n}) \in {\mathbb R}^{n \times n}$  with positive entries $d_{1}, \ldots , d_{n}$
and ${\mathbf x} \in {\mathbb R}^{n}$, we can shift factors  from one affine linear mapping in (\ref{DNN})  to the next,  i.e.,
$$ {\mathbf F}_{\ell}(t) = {\mathbf A}^{(\ell)} \sigma( {\mathbf F}_{\ell-1}(t)) + t {\mathbf c}^{(\ell)} + {\mathbf b}^{(\ell)} =
{\mathbf A}^{(\ell)} {\mathbf D} \sigma( {\mathbf D}^{-1} {\mathbf F}_{\ell-1}(t)) + t {\mathbf c}^{(\ell)} + {\mathbf b}^{(\ell)} ,
$$
see e.g. \cite{Phuong20, Bona21}. This positive scaling property can  be used for  stabilization  of the DNN, see \cite{Stock19}. 
Phuong and Lampert \cite{Phuong20} have  shown that  under certain conditions  on the structure of the DNN (as non-vanishing parameters, decaying width etc.) there are no  other function-preserving parameter transforms besides positive scaling  and permutation. The problem of identifiability  has also been  studied in \cite{Bona21}, where some assumptions in \cite{Phuong20} have been relaxed. 

The results in \cite{Bona21} and \cite{Phuong20}  also implicitly give a bound  on the number of independent  parameters  in a ReLU DNN. Here, we say that  the set of  parameters determining a model is \textit{independent}, if  the reduced model being obtained by pre-determining  one of the  parameters to be a fixed constant, does not longer represent all functions that can be represented by the original model.

In this paper, we  are interested in a better understanding of the ReLU DNN model (\ref{DNN}) in comparison to the CPL spline model (\ref{sp}). While the shallow ReLU NN is equivalent to the CPL spline model, we will show in detail, how it starts to differ for more hidden layers. 
We will derive a precise relation between the parameter set  determining the ReLU DNN and the parameter set  defining the CPL spline function.
In particular, we will construct ReLU DNN functions with a maximal number of arbitrarily prescribed breakpoints. These observations also imply that, after removing redundancies due to positive scaling, the parameter set determining  the ReLU DNN model (\ref{DNN}) is independent.  Our approach to detect the positive scaling property as the only reason for parameter redundancies  strongly differs from those in \cite{Bona21,Phuong20},  and our model is not covered  by the considerations in \cite{Bona21,Phuong20}.
 We show the independence of the set of parameters by rephrasing the model as a CPL spline function with a large number of independent (active) breakpoints. Here we say that a breakpoint of $f \in \Sigma_{N}$ is \textit{active} if the corresponding coefficient in (\ref{sp}) does not vanish. 

The obtained structure of the model nicely shows how the breakpoints corresponding to the different layers of the ReLU DNN model interlace, such that the first layer breakpoints provide a coarse grid that can be refined or extended by the breakpoints corresponding to the further layers.  
Our observations  on the relation between the  ReLU DNN model and the CPL spline model can for example be used to set a priori conditions to pre-determine special breakpoints or breakpoint sets  as well as function values  at intermediate layers or in the output of the ReLU DNN.

This paper is structured as follows.
In Section \ref{sec:shallow} we will briefly summarize  the ReLU shallow  network  in the case of  a given  source channel. It turns out  that  in this case  we have ${\mathcal Y}_{n}=\Sigma_{n}$. 

In Section \ref{sec:two}, we study the ReLU DNN  with two hidden layers  ($L=3$)  in more detail. In Theorem \ref{theo1} we give a constructive procedure to transfer the parameter set determining $f_{3} \in {\mathcal Y}_{n_{1},n_{2}}$ to the parameter set of its representation in $\Sigma_{N}$ with $N \le N_{3,\max} :=(n_{1}+1)(n_{2}+1)-1$. This procedure is used  in Section \ref{sec:break}  to construct  a function $f_{3} \in {\mathcal Y}_{n_{1},n_{2}}$ with the maximal number of $N_{3,\max}$ arbitrarily prescribed breakpoints, see Theorems \ref{theomax} and \ref{theomax1}. 
Moreover,  we can conclude that the set of parameters of ${\mathcal Y}_{n_{1},n_{2}}$ is  independent after normalization ${\mathbf c}^{(2)} = \sign({\mathbf c}^{(2)})$. Since ${\mathcal Y}_{n_{1},n_{2}}$ 
possesses after this normalization $n_{1}n_{2}+ n_{1}+ 2n_{2}+2$ real parameters (and one sign vector), the coefficients in the spline representation of $f_{3}$ depend therefore essentially on the breakpoints.

In Section \ref{sec:more}, we extend  these results to  ReLU DNN with $L$ layers.  
We will show that each function $f_{L}$ in (\ref{DNN}) admits a representation as a CPL spline function,
 $${\mathcal Y}_{n_{1}, n_{2}, \ldots , n_{L-1}} \subseteq \,  \Sigma_{N_{L,\max}} \qquad \text{with} \qquad  N_{L,\max} := \prod\limits_{\ell=1}^{L-1} (n_{\ell}+1) -1, $$  where the upper bound  $N_{L,\max}$ is sharp, see also \cite{Serra18}. Moreover, we provide a recursive algorithm to transfer the parameter set of a function $f_{L}  \in {\mathcal Y}_{n_{1}, n_{2}, \ldots , n_{L-1}}$   in (\ref{DNN}) into a the parameter set of  $f_{L} \in \Sigma_{N}$, and in particular, we can compute all breakpoints of $f_{L}$, see Algorithm \ref{algo2}.
  In Section \ref{sec3knots} we show  for three hidden layers, how to construct a function $f_{4} \in {\mathcal Y}_{n_{1},n_{2}, n_{3}}$ with $n_{1}n_{2} + n_{2}n_{3} + n_{1} + n_{2}+ n_{3}$ arbitrarily prescribed breakpoints. As before, this result shows that the parameter set determining ${\mathcal Y}_{n_{1},n_{2}, n_{3}}$ is independent after normalization ${\mathbf c}^{(\ell)} = \sign({\mathbf c}^{(\ell)}) \in \{-1,0,1\}^{n_{\ell}}$ for $\ell=2,3$.

%%%%%%%%%%%%%%%%%%%%%%%%%%%%%%%%%%%%%%%%%%%%%%%%%%%%%%%%%%%%%%%%%%%%%%%%%%%%%%%%%%%%%%%%%%%%%%%%%%%%%%%%
%%%%%%%%%%%%%%%%%%%%%%%%%%%%%%%%%%%%%%%%%%%%%%%%%%%%%%%%%%%%%%%%%%%%%%%%%%%%%%%%%%%%%%%%%%%%%%%%%%%%%%%%
\section{ReLU Shallow Network}
\setcounter{equation}{0}
\label{sec:shallow}

In the special case of only one hidden layer of width $n=n_{1}$ with a single input and a single output, the ReLU NN ${\mathcal Y}_{n}$ in (\ref{DNN}) has the form
\begin{equation}\label{lone}
 f_{2}(t) = c^{(2)}t + b^{(2)}  + {\mathbf A}^{(2)} \, \sigma (  t\, {\mathbf A}^{(1)} + {\mathbf b}^{(1)} )  = c^{(2)} t + b^{(2)} + \sum_{j=1}^{n} a_{j}^{(2)} \, \sigma(a_{j}^{(1)} t + b^{(1)}_{j}) 
\end{equation} 
with the parameter set $\Lambda_{1}= \{{\mathbf A}^{(1)}, \, {\mathbf A}^{(2)},  \, c^{(2)}, \, {\mathbf b}^{(1)}, b^{(2)}\} \subset {\mathbb R}^{3n+2}$ with  $c^{(2)}, \, b^{(2)} \in {\mathbb R}$, and ${\mathbf A}^{(1)}=(a_{j}^{(1)})_{j=1}^{n}, \, {\mathbf A}^{(2)}=((a_{j}^{(2)})_{j=1}^{n})^{T}, \,{\mathbf b}^{(1)}=(b^{(1)}_{j})_{j=1}^{n}$. In this setting, the hidden layer has $n$ units.
We briefly recall the expressivity  and  parameters redundancies of this model, see e.g.\ \cite{Daub22}. 

\begin{lemma}\label{lemma1}
The function models $\Sigma_{n}$ and ${\mathcal Y}_{n}$ in $(\ref{sp})$  and $(\ref{lone})$ are equivalent, % i.e., model $f_{2}(t)$ in $(\ref{lone})$ is equivalent to the model $s_{N}(t)$ given in $(\ref{sp})$ for $N=n$, 
i.e., any function $f_{2}$ in $(\ref{lone})$ can be represented in the form $(\ref{sp})$, and any \textnormal{CPL} spline function $s_{n}$ in $(\ref{sp})$ (with $N=n$) can be represented as a function $f_{2}$ in $(\ref{lone})$.
Moreover, $f_2$ is a \textnormal{CPL} spline with exactly $n$ active breakpoints, if $a_j^{(2)} |a_j^{(1)} |\neq 0$ for $j=1, \ldots , n$, and  if  $-(b_{j}^{(1)}/a_j^{(1)})$ are  for $j=1, \ldots , n$  pairwise distinct.
%\begin{equation}\label{simp1} s_{N}(t) =  {q}_{1}(t) + \sum_{\ell=1}^{k} \alpha^{(2)}_{\ell} \, \sigma(t - x_{\ell}),
%\end{equation}
%where ${q}_{1}(t) = {q}_{1}t + q_{0}$ is a linear polynomial, $k\le n$ is the number of pairwise distinct bounded parameters $-b_{j}^{(1)}/a_{j}^{(1)}$ in $(\ref{lone})$, $(\tilde{a}^{(2)}_{\ell})_{\ell=1}^{k} \in {\mathbb R}^{k}$, and 
%$x_{1} <  x_{2} <  \ldots <  x_{k}$. 
In particular, $f_{2}$ in $(\ref{lone})$ is a \textnormal{CPL} spline function in $\Sigma_{n}$ with at most $n$ knots, and is completely determined by at most $2n+2$ parameters.
\end{lemma}

\begin{proof}
For any $a,b \in {\mathbb R}$ we observe that 
$$ \sigma(at+b) =  \begin{cases}  |a| \, \sigma(t + \frac{b}{a}) & a > 0 ,\\
\sigma(b) & a=0, \\
(at+b) + |a|\sigma(t+\frac{b}{a}) & a <0.
\end{cases}  $$
Therefore $f_{2}$ in (\ref{lone}) can be represented as
\begin{align}\nonumber
f_{2}(t) &=  c^{(2)} t + b^{(2)} +  \sum\limits_{\substack{j=1\\ a_{j}^{(1)} > 0}}^{n} a_{j}^{(2)} |a_{j}^{(1)}| \sigma\Big(t+\frac{b_{j}^{(1)}}{a_{j}^{(1)}} 
\Big) \\
& \qquad + \sum\limits_{\substack{j=1\\ a_{j}^{(1)} < 0}}^{n} a_{j}^{(2)} \Big(|a_{j}^{(1)}| \sigma\Big(t+\frac{b_{j}^{(1)}}{a_{j}^{(1)}} 
\Big) + a_{j}^{(1)} t + b_{j}^{(1)} \Big) 
 +\sum\limits_{\substack{j=1\\ a_{j}^{(1)} = 0}}^{n}  a_{j}^{(2)} \sigma(b_{j}^{(1)})  \\
\label{star}
&= \textstyle q(t) + \sum\limits_{\substack{j=1\\  a_{j}^{(1)} \neq 0}}^{n} a_{j}^{(2)} |a_{j}^{(1)}| \sigma\Big(t+ \frac{b_{j}^{(1)}}{a_{j}^{(1)}} 
\Big)  
\end{align}
with a linear polynomial 
\begin{equation}\label{q}
q(t) = \textstyle (c^{(2)} t + b^{(2)}) + \sum\limits_{\substack{j=1\\ a_{j}^{(1)} < 0}}^{n} a_{j}^{(2)} \Big(  a_{j}^{(1)} t + b_{j}^{(1)}\Big)  +\sum\limits_{\substack{j=1\\ a_{j}^{(1)} = 0}}^{n}  a_{j}^{(2)} \sigma(b_{j}^{(1)}). 
\end{equation}
We  apply a simplification of (\ref{star}), if $a_{j}^{(1)} = 0$, $a_{j}^{(2)} = 0$, or if coinciding values $\frac{b_{j}^{(1)}}{a_{j}^{(1)}}$ appear, and then a re-indexing such that the remaining values $x_{j} := -\frac{b_{j}^{(1)}}{a_{j}^{(1)}}$  are ordered by size, and arrive  at a spline function representation as in (\ref{sp}) with at most $n$ knots, i.e., $f_{2} \in \Sigma_{n}$.
Conversely, $s_{n}(t)$ in model (\ref{sp}) is a special case of (\ref{lone}) taking for example $q(t) = c^{(2)} t + b^{(2)}$, $a_{j}^{(2)} := \alpha_{j}$, $a_{j}^{(1)} :=1$,  $b_{j}^{(1)} := -x_{j}$ for $j=1, \ldots, n$.
Obviously, $s_{n}$ in (\ref{sp}) possesses  $2n+2$ parameters.
\end{proof}

\begin{remark} \label{rem11}
The transform from model (\ref{lone}) to (\ref{sp}) obviously covers also the well-known redundancy caused by positive scaling, see e.g.\ \cite{Phuong20}.
For a diagonal matrix ${\mathbf D}$ with positive weights, we always have 
$$ {\mathbf A}^{(2)} {\mathbf D} \sigma(t {\mathbf A}^{(1)} + {\mathbf b}^{(1)}) = {\mathbf A}^{(2)} \sigma( t  {\mathbf D}{\mathbf A}^{(1)}+  {\mathbf D}{\mathbf b}^{(1)}). $$
Assuming that all components of ${\mathbf A}^{(1)}$ do not vanish, we can take ${\mathbf D}= \diag (|a_{j}^{(1)}|^{-1})_{j=1}^{n}$ and therefore simplify the model (\ref{lone}) to an equivalent model, where ${\mathbf A}^{(1)}$ is replaced by a sign vector in $\{-1,1\}^{n}$. If we employ a source channel, then we can actually replace ${\mathbf A}^{(1)}$ by ${\mathbf 1}$, as shown in Lemma \ref{lemma1}.
\end{remark}

We summarize the transfer from model (\ref{lone}) to (\ref{sp}) in Algorithm \ref{alg1}.

\begin{algorithm}[ht]\caption{Transfer from (\ref{lone}) to (\ref{sp})}
\label{alg1}
\small{
\textbf{Input:} $c^{(2)}, \, b^{(2)} \in {\mathbb R}$,  ${\mathbf A}^{(1)}=(a_{j}^{(1)})_{j=1}^{n}, \, {\mathbf A}^{(2)}=((a_{j}^{(2)})_{j=1}^{n})^{T}, \,{\mathbf b}^{(1)}=(b^{(1)}_{j})_{j=1}^{n} \in {\mathbb R}^{n}$ in (\ref{lone}).

\begin{description}
\item{1.} Initialize $q_{1}:= c^{(2)}$, $q_{0}:= b^{(2)}$, $\balpha := \mathrm{zeros}(n)$, ${\mathbf x} := \mathrm{zeros}(n)$, $N_{2}:=n$. 
\item{2.} for $j=1:n$ do\\
$\alpha_j := a_j^{(2)} |a_j^{(1)}|;$ \\
if $a_j^{(1)} <0$ then $x_{j}:= -\frac{b_{j}^{(1)}}{a_{j}^{(1)}}$; $q_{1}:=q_{1}+a_{j}^{(1)}a_{j}^{(2)}$;
$q_{0}:=q_{0}+b_{j}^{(1)}a_{j}^{(2)}$; end(if)\\
if $a_j^{(1)} >0$ then $x_{j} := -\frac{b_{j}^{(1)}}{a_{j}^{(1)}}$; end(if)\\
if $a_j^{(1)} ==0$ then $N_{2}:=N_{2}-1$; $q_{0} :=q_{0}+a_{j}^{(2)} \max \{ b_{j}^{(1)},0\}$; remove $x_{j}$ from ${\mathbf x}$, remove $\alpha_{j}$ from $\balpha$; end(if) \\
\hspace*{-4mm} end(for($j$))
\item{3.} Apply a permutation such that the components of ${\mathbf x}$ are ordered by size, $x_{1}\le x_{2} \le \ldots \le x_{N_{2}}$. Use the same permutation to order the corresponding coefficient vector $\balpha$. \\
%$[{\mathbf y}, \mathbf{ind}] = \textrm{sort}({\mathbf x})$;\\
%$\balpha= \textrm{sort}(\balpha, \mathbf{ind})$;\\
%${\mathbf x} = {\mathbf y}$; \\
for $j=2:N_{2}$ do\\
\null \quad if $x_{j} ==x_{j-1}$ then $\alpha_{j-1}:=\alpha_{j-1}+\alpha_{j}$; $N_{2}:=N_{2}-1$; remove $x_{j}$ from ${\mathbf x}$ and $\alpha_{j}$ from $\balpha$; \\
\null \quad end(if)\\
 end(for($j$))
\end{description}

\textbf{Output:} $q_{0}, \, q_{1} \in {\mathbb R}$, $\balpha =(\alpha_{\ell})_{\ell=1}^{N_{2}}$, ${\mathbf x} = (x_{\ell})_{\ell=1}^{N_{2}}$ determining a CPL spline function with $N_{2}$ knots.}
\end{algorithm}

\section{ReLU NN for two hidden layers}
\label{sec:two}

We consider now the model ${\mathcal Y}_{n_1,n_2}$ of functions $f_{3}: {\mathbb R} \to {\mathbb R}$ with three layers (two hidden layers)  of the form 
\begin{align} \nonumber
 f_{3}(t) &= (c^{(3)} t + b^{(3)}) + {\mathbf A}^{(3)} \, \sigma \Big( {\mathbf A}^{(2)}  \Big(\sigma ( t {\mathbf A}^{(1)}  + {\mathbf b}^{(1)}) \Big) + t {\mathbf c}^{(2)} + {\mathbf b}^{(2)} \Big)  \\
 \label{twol}
 &= \textstyle c^{(3)} t + b^{(3)} + \sum\limits_{j=1}^{n_{2}} a^{(3)}_{j} \sigma\left( \sum\limits_{k=1}^{n_{1}} a^{(2)}_{j,k} \, \sigma(a^{(1)}_{k} t + b^{(1)}_{k}) + c_{j}^{(2)} t +b^{(2)}_{j} \right) ,
 \end{align}
where $f_{3}$ depends on the parameter set 
$$\Lambda_{2} = \{ c^{(3)}, b^{(3)}, \,   {\mathbf A}^{(1)}, {\mathbf A}^{(2)}, {\mathbf A}^{(3)}, {\mathbf b}^{(1)}, {\mathbf c}^{(2)}, {\mathbf b}^{(2)} \} \subset {\mathbb R}^{n_{1} n_{2}+2 n_{1} + 3n_{2}+2}$$ 
with
${\mathbf A}^{(1)} = (a_{k}^{(1)})_{k=1}^{n_{1}} \in {\mathbb R}^{n_{1} \times 1}$,  ${\mathbf b}^{(1)}=(b_{k}^{(1)})_{k=1}^{n_{1}} \in {\mathbb R}^{n_{1}}$, ${\mathbf A}^{(3)} =((a_{j}^{(3)})_{j=1}^{n_{2}})^{T}\in {\mathbb R}^{1 \times n_{2}}$, ${\mathbf c}^{(2)} = (c_{j}^{(2)})_{j=1}^{n_{2}}, {\mathbf b}^{(2)} = (b_{j}^{(2)})_{j=1}^{n_{2}} \in {\mathbb R}^{n_{2}}$, 
${\mathbf A}^{(2)} = (a^{(2)}_{j,k})_{j,k=1}^{n_{2},n_{1}} \in {\mathbb R}^{{n_2 \times n_1}}$, and $c^{(3)}, b^{(3)} \in {\mathbb R}$. 

\noindent
\textbf{Assumption:}
In this section, we assume that ${\mathbf A}^{(1)}$ has no vanishing entries and that $-\frac{b_{k}^{(1)}}{a_{k}^{(1)}}$ are pairwise distinct, since otherwise, by Lemma \ref{lemma1}, the first layer has a true width being smaller than $n_{1}$.

\subsection{Representation  as a continuous linear spline function model} 

We start with the following observation that will be an essential tool for our further investigations of the structure of ReLU NN with more than one hidden layers.

\begin{lemma}\label{lemsigma}
For a given $\textnormal{CPL}$ spline function 
$ f(t) =  ct + b + \sum\limits_{k=1}^{n} \alpha_{k} \, \sigma(t - x_{k})  \in \Sigma_{n}$ with $x_{1}< x_{2}< \ldots < x_{n}$
we have 
\begin{equation} \label{sigf}
(\sigma \circ f)(t) = \textstyle \sigma(f(t)) = \tilde{c}t + \tilde{b} + \sum\limits_{k=1}^{n} \tilde{\alpha}_{k} \, \sigma(t - x_{k}) + \sum\limits_{\nu=0}^{n} \beta_{\nu}  \, \sigma(t - \tilde{x}_{\nu}), 
\end{equation}
where, with $\chi_{T}(t)$ denoting the characteristic function of the subset $T \subset {\mathbb R}$, 
\begin{align*}
\tilde{x}_{\nu} &= \left\{ \begin{array}{ll}
- \frac{\eta_{\nu}}{\mu_{\nu}} & \mu_{\nu} \neq 0,\\
 - \infty & \mu_{\nu} =0, \end{array} \right.\\
	\tilde{c} &= - \sigma(-c),\\
	\tilde{b} &= b \, \chi_{(-\infty,0)}(c)+ \sigma(b) \, \chi_{\{0\}}(c),\\
	\tilde{\alpha}_k &=  \alpha_{k} \, \chi_{(0,\infty)}(f(x_{k}))  + (\sigma(\mu_k)+\sigma(-\mu_{k-1})) \, \chi_{\{0\}}(f(x_{k})),\\
	\beta_\nu &= |\mu_\nu| \, \chi_{(x_\nu,x_{\nu+1})}(\tilde{x}_\nu).
\end{align*}
Here, $\mu_{k}:= c+ \sum\limits_{\ell=1}^{k} \alpha_{\ell}$ and $\eta_{k} := b - \sum\limits_{\ell=1}^{k} \alpha_{\ell} x_{\ell}$. In particular, $\sigma \circ f \in \Sigma_{2n+1}$, i.e., $\sigma \circ f$ possesses at most $2n+1 $ breakpoints.
\end{lemma}

\begin{proof}
1. Let $x_{0}:=-\infty$ and $x_{n+1} := \infty$. 
The CPL spline  function $f$ %is differentiable at ${\mathbb R} \setminus \{ x_{1}, \ldots , x_{n}\}$  and we have 
has for $t \in (x_{k}, x_{k+1})$ the form
$$ f(t) = \mu_{k} t + \eta_{k}, \qquad k=0, \ldots, n, $$
where $\mu_{0}:= c$, $\eta_{0} := b$, 
\begin{align*} \textstyle \mu_{k} := \mu_{k-1} + \alpha_{k} = c + \sum\limits_{\ell=1}^{k} \alpha_{\ell}, \quad 
\eta_{k} := \eta_{k-1} - \alpha_{k}x_{k} = b - \sum\limits_{\ell=1}^{k} \alpha_{\ell} x_{\ell}, \quad k=1, \ldots , n.
\end{align*}
In other words, if $\partial_{+}$ and $\partial_{-}$ denote the right-sided and left-sided derivatives, respectively, then
$$ \partial_{-} f(x_{1}) = \mu_{0}, \quad \partial_{+} f(x_{n}) = \mu_{n}, \quad \partial_{+} f(x_{k}) = \partial_{-}f(x_{k+1}) = \mu_{k}, \quad k=1, \ldots , n-1. $$
Furthermore, the coefficients in the representation of $f$ satisfy $\alpha_{k} = \mu_{k}- \mu_{k-1} = \partial_{+} f(x_{k}) - \partial_{-} f(x_{k})$. 

2. Obviously, $\sigma \circ f$ is again a CPL spline function. The goal is to find a representation of $\sigma \circ f$  as a function in $\Sigma_{N}$ for suitable $N \in {\mathbb N}$.
The definition of $\sigma(t)$ implies that a value $x \in {\mathbb R}$ can only be  a breakpoint of $\sigma \circ f$, if it is already a breakpoint of $f$, i.e., $x \in \{x_{1}, \ldots , x_{n}\}$, or if it is a singular zero of $f$, i.e., $f(x)=0$ and $x$ is not inside an open interval, where $f$ is constantly vanishing. \\
Since $f$ is linear on each interval $(x_{k}, x_{k+1})$, $k=0, \ldots , n$, $f$ has at most one zero in $(x_{k}, x_{k+1})$ if $\mu_{k} \neq 0$,  and this zero $\tilde{x}_{k}$ satisfies $f(\tilde{x}_{k}) = \mu_{k} \tilde{x}_{k} + \eta_{k} =0$. 
Thus, the possible new breakpoints are found as 
$ \tilde{x}_{k} :=
- \frac{\eta_{k}}{\mu_{k}}$ for $\mu_{k} \neq 0$. For 
$\mu_{k} =0$, the function $f$ is constant in $(x_{k}, x_{k+1})$. In this case $\sigma(f(t)) = \max \{ f(t), 0 \}$ is also constant in $(x_{k}, x_{k+1})$, and we do not get a new breakpoint.  To simplify the notation, we denote $\tilde{x}_{k} := - \infty$ in this case.
Thus, we can write $\sigma \circ f$ in the form (\ref{sigf}) with $\tilde{x}_{\nu}$ as given in Lemma \ref{lemsigma}, and where we still need to determine $\tilde{c}$, $\tilde{b}$, $\tilde{\alpha}_{k}$, $\beta_{\nu}$. 

3. For the parameters $\tilde{c}$ and $\tilde{b}$ we find
$$ \tilde{c} = \lim_{t  \to - \infty} \partial(\sigma(f(t))) = \left\{ \begin{array}{ll}  c & c < 0,\\ 0 & c \ge 0, \end{array} \right.
\quad \tilde{b} = \lim_{t  \to - \infty} (\sigma(f(t)) - \tilde{c} t) = \left\{ \begin{array}{ll} b & c < 0,\\
\sigma(b) & c =0, \\
0 & c >0. \end{array} \right.
$$

\begin{figure}[h]
\begin{center}
	\includegraphics[scale=0.3]{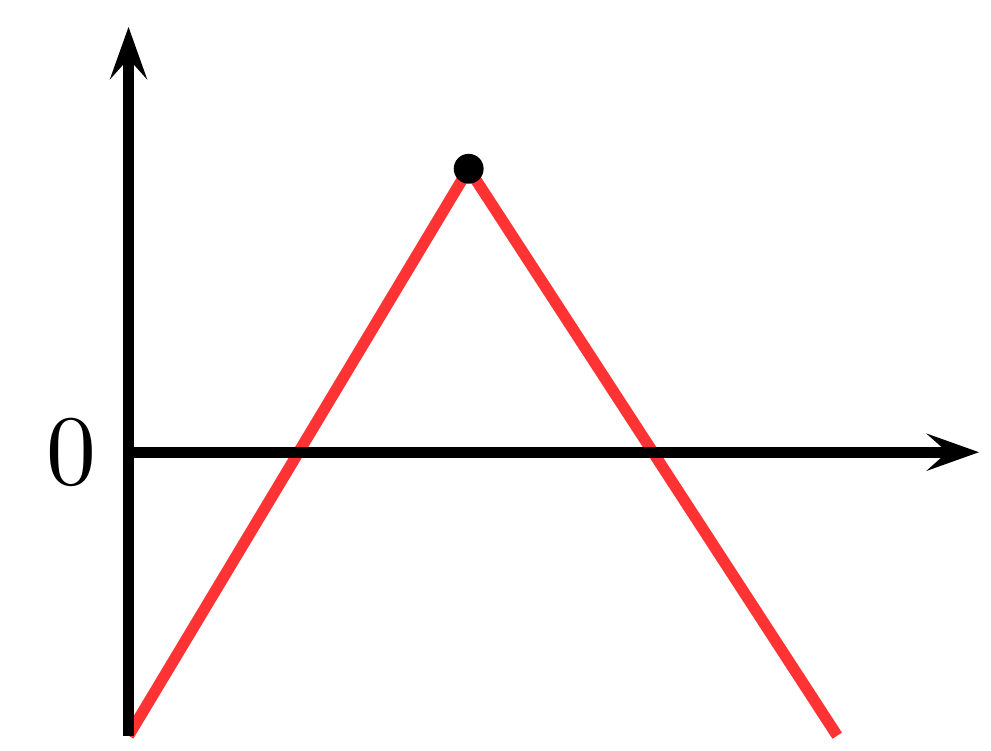}
	\includegraphics[scale=0.3]{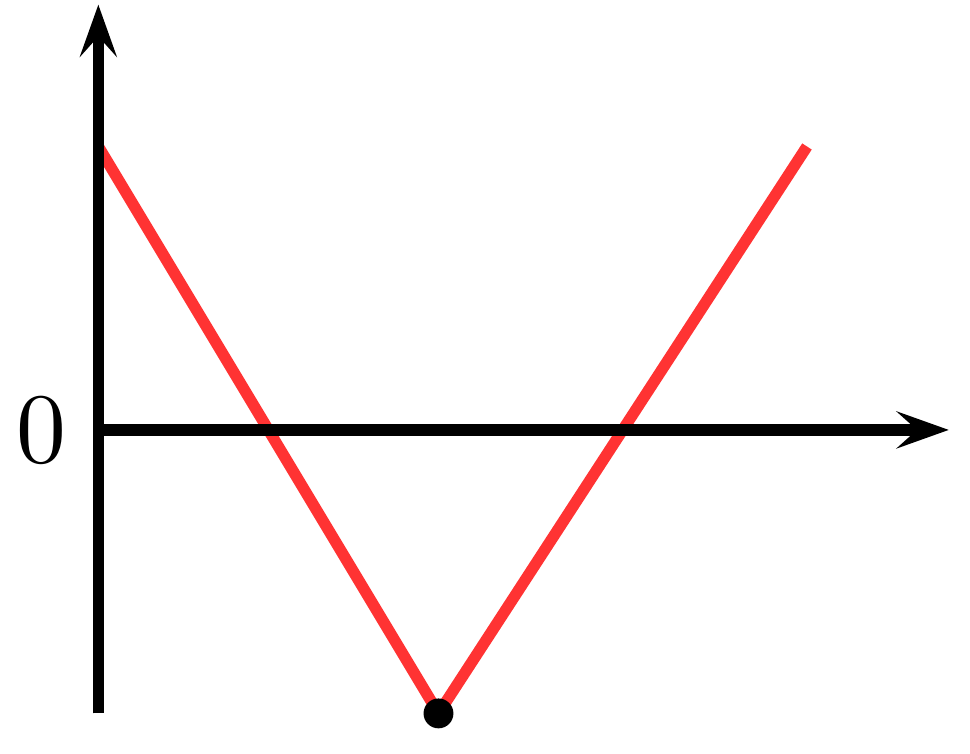}\\
	\includegraphics[scale=0.3]{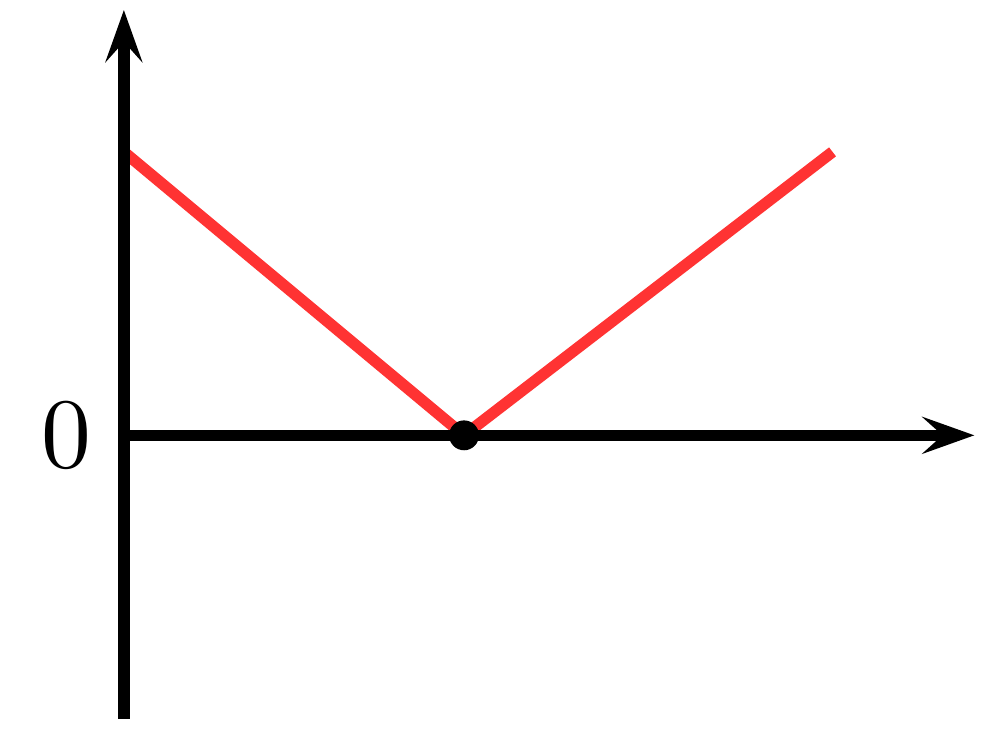}
	\includegraphics[scale=0.3]{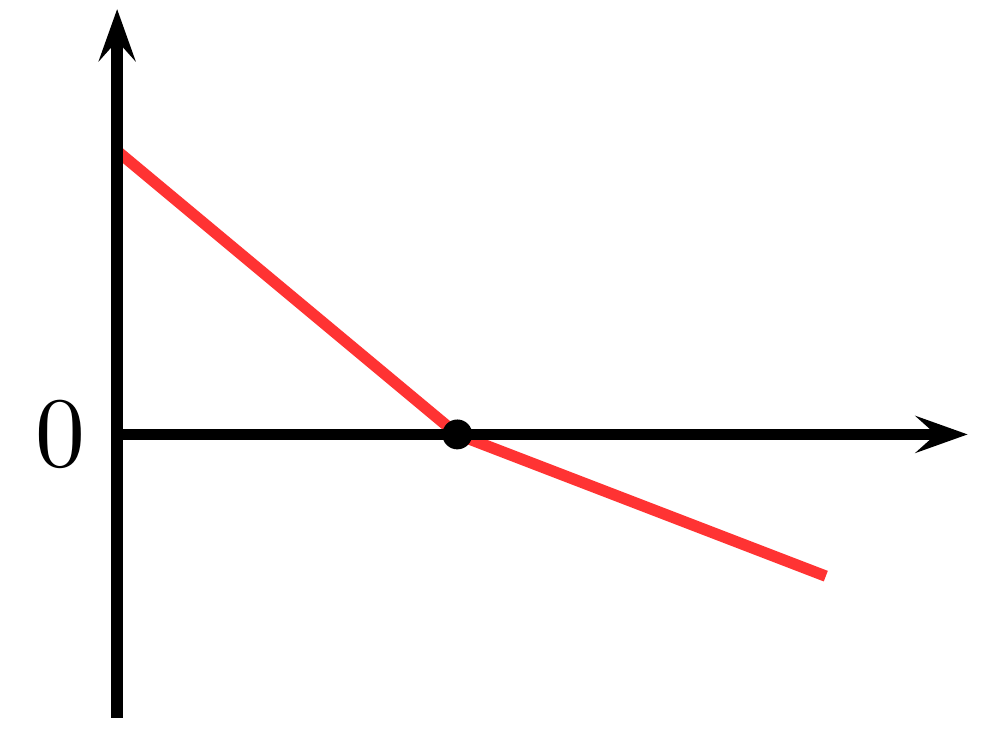}
	\includegraphics[scale=0.3]{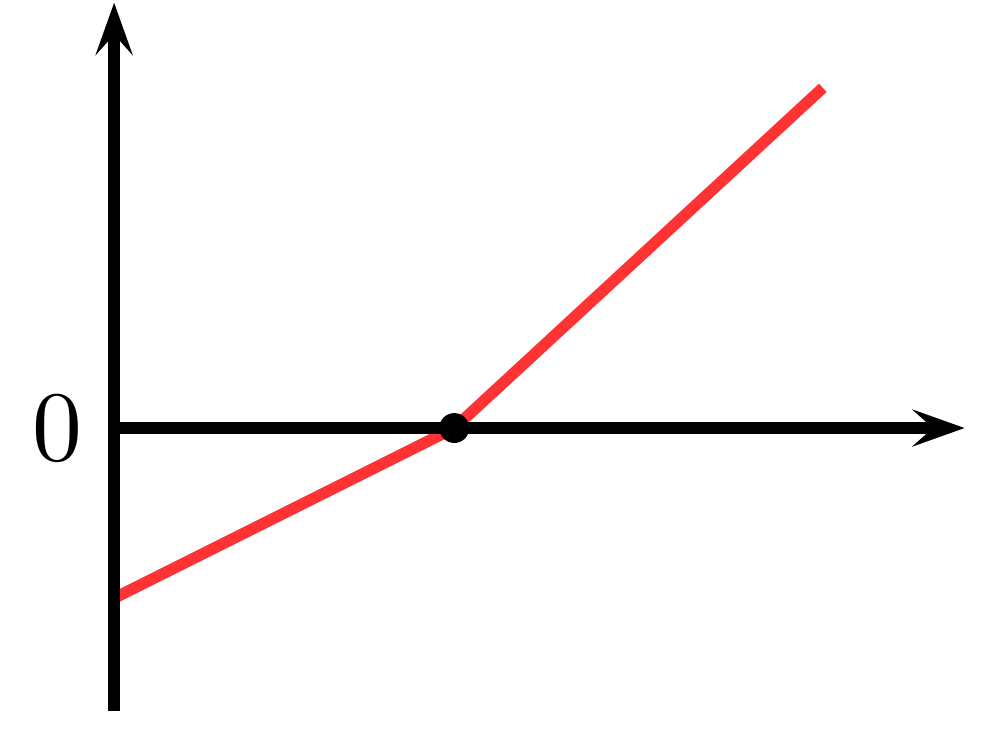}
	\includegraphics[scale=0.3]{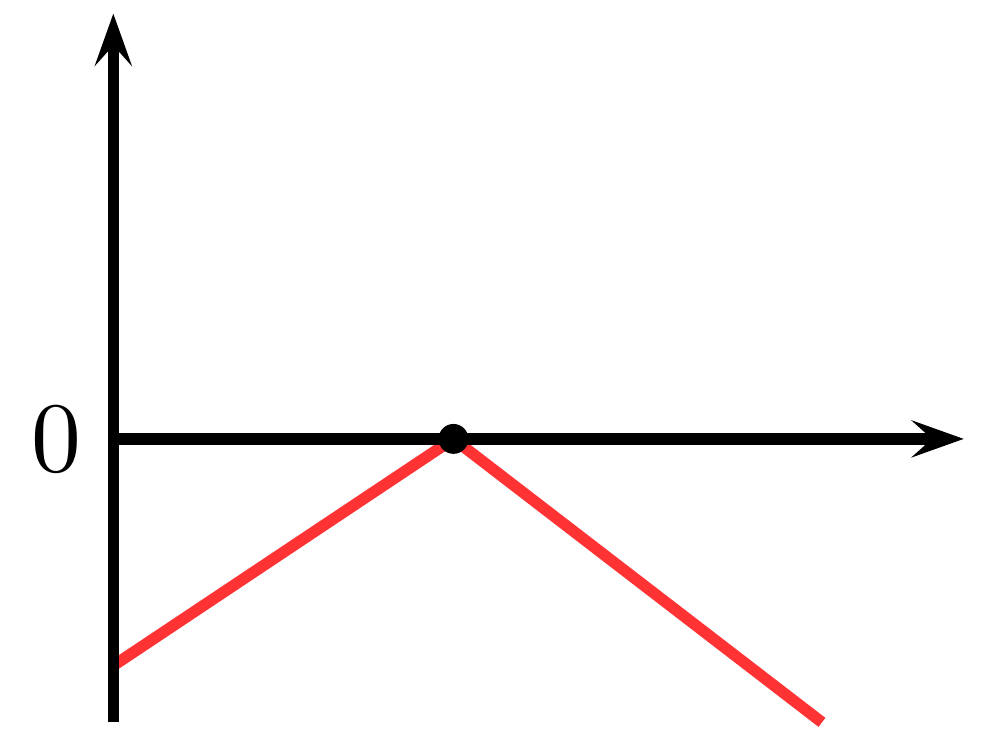}
\end{center}
\caption{\small{Illustration of different situations for breakpoints $f(x_{k})$.\\ First row: $f(x_{k}) = \sigma(f(x_{k})) >0$, and $f(x_{k})<0$ with $\sigma(f(x_{k})) =0$.\\ Second row: Four cases of slopes $\partial_{+}f(x_{k})$ and $\partial_{-}f(x_{k})$.}}
\label{fig0}
\end{figure}

Now, we determine $\tilde{\alpha}_{k}$. If $f(x_{k})>0$,  then the continuity of $f$ implies that $(\sigma \circ f)(t) = f(t)$ in an $\epsilon$-neighborhood of $x_{k}$ and therefore $\tilde{\alpha}_{k}  = \alpha_{k}$. 
If $f(x_{k})<0$, then $(\sigma \circ f)(t) = 0$ in an $\epsilon$-neighborhood of $x_{k}$ and therefore $\tilde{\alpha}_{k}=0$.
Finally, if $f(x_{k})=0$, we consider the left and right $\epsilon$-neighborhood of $x_{k}$. We obtain with $\partial_{+}f(x_{k}) = \mu_{k}$ and $\partial_{-}f(x_{k}) = \mu_{k-1}$ that 
$$ \partial_{+}(\sigma \circ f)(x_{k}) = \sigma(\mu_{k}),
%\left\{ \begin{array}{ll}  \mu_{k} &  \mu_{k} >0, \\
%0 & \mu_{k} \le 0,  \end{array} \right. 
\qquad \partial_{-}(\sigma \circ f)(x_{k}) = -\sigma(-\mu_{k-1}), 
%\left\{ \begin{array}{ll}  \mu_{k-1} &  \mu_{k-1} <0, \\
%0 & \mu_{k-1} \ge 0,  \end{array} \right. 
$$
and therefore $\tilde{\alpha}_{k} = \partial_{+}(\sigma \circ f)(x_{k})-\partial_{-}(\sigma \circ f)(x_{k}) = \sigma(\mu_k)+\sigma(-\mu_{k-1})$, see Figure \ref{fig0}.

Finally, we determine $\beta_{\nu}$.
If  $\tilde{x}_{k} \notin (x_{k}, x_{k+1})$, then $f$ does not have a singular zero in $(x_{k}, x_{k+1})$, i.e., $\tilde{x}_{k}$ is not a new breakpoint of $\sigma \circ f$ and  $\beta_{k}=0$. We say that this ``breakpoint'' is not active.
In particular, $\tilde{x}_{k} =- \infty$ is never an active breakpoint. If $\tilde{x}_{k} \in (x_{k}, x_{k+1})$, then 
$$ \partial_{+} ( \sigma \circ f)(\tilde{x}_{k}) = \sigma(\mu_{k}), \qquad \partial_{-} ( \sigma \circ f)(\tilde{x}_{k}) = -\sigma(-\mu_{k}), $$
and thus  $\beta_{k} = |\mu_{k}|$. 
\end{proof}

\begin{remark}
While Lemma \ref{lemsigma} shows that $\sigma \circ f \in \Sigma_{2n+1}$ for $f \in \Sigma_{n}$,  the representation (\ref{sigf})  contains for $n > 1$ always breakpoints that are not active. The largest number of breakpoints appears if $f$ possesses a singular zero $\tilde{x}_{k}$ in each interval $(x_{k}, x_{k+1})$, $k=0, \ldots , n$. But in this case, either all function values $f(x_{2k})$ or all function values $f(x_{2k+1})$ are negative, i.e., either all $x_{2k}$ or all $x_{2k+1}$ are not longer active breakpoints in $\sigma \circ f$, such that $\sigma \circ f$ has at most $n+1 + \lfloor \frac{n+1}{2} \rfloor$ active breakpoints.
\end{remark}

With these preliminaries we now study  the expressivity of the model ${\mathcal Y}_{n_{1},n_{2}}$ and show in the following theorem, how $f_{3} \in {\mathcal Y}_{n_{1},n_{2}}$ can be represented as a CPL spline function in $\Sigma_{N}$ with $N=(n_{1}+1)(n_{2}+1)-1$.

\begin{theorem}\label{theo1}
For $n_{1} \ge 1$ and $n_{2} \ge 2$ the function $f_{3} \in {\mathcal Y}_{n_{1},n_{2}}$ in  $(\ref{twol})$ 
can be represented  as
\begin{equation}\label{mo2}
{f}_{3}(t) = q_{1}t + q_{0} + \sum_{\ell=1}^{n_{1}} \alpha_{\ell} \, \sigma(t-x_{\ell}) + \sum_{j=1}^{n_{2}}\sum_{\nu=0}^{n_{1}} \alpha_{j,\nu} \, \sigma(t-x_{j,\nu}), 
\end{equation}
where $q_{0,} q_{1}, \alpha_{\ell}, x_{\ell}, \alpha_{j,\nu}, \, x_{j,\nu}$ are real parameters with 
$ -\infty < x_{1} < x_{2} < \ldots  < x_{n_{1}} < \infty$.
Furthermore,  all active breakpoints $x_{j,\nu}$ (i.e., breakpoints with $\alpha_{j,\nu} \neq 0$) satisfy $x_{j,\nu} \in (x_{\nu}, x_{\nu+1})$ for $j=1, \ldots , n_{2}$, $\nu=0, \ldots, n_{1}$, with the convention $x_{0} := -\infty$ and $x_{n_{1}+1} := \infty$. 
\end{theorem}

\begin{proof} 1.
We employ the notation 
$$ f_{2,j}(t) := \sum_{k=1}^{n_{1}} a_{j,k}^{(2)} \sigma(a_{k}^{(1)} t + b_{k}^{(1)}) +c_{j}^{(2)} t  + b_{j}^{(2)}, \quad j=1, \ldots , n_{2}, 
$$
then $f_{3}$ in (\ref{twol}) reads
\begin{equation}\label{f3mitf2} f_{3}(t) = c^{(3)}t + b^{(3)}+ \sum_{j=1}^{n_{2}} a_{j}^{(3)} \, \sigma(f_{2,j}(t)). 
\end{equation}
 Lemma  \ref{lemma1}, we can always rewrite $f_{2,j}(t)$ as 
\begin{equation}\label{f2j} f_{2,j}(t) = \tilde{c}_{j}^{(2)} t + \tilde{b}_{j}^{(2)} + \sum_{k=1}^{n_{1}} \tilde{a}_{j,k}^{(2)} \sigma(t- x_{k}),
\end{equation}
where  $\tilde{c}_{j}^{(2)} $ and $\tilde{b}_{j}^{(2)} $  are determined as in (\ref{q}), i.e.,
\begin{equation}\label{bc}
\tilde{c}_{j}^{(2)} = c_{j}^{(2)} - \sum_{k=1}^{n_{1}} a_{j,k}^{(2)} \, \sigma(-a_{k}^{(1)}), \qquad 
\tilde{b}_{j}^{(2)} = b_{j}^{(2)} + \sum_{k=1}^{n_{1}} a_{j,k}^{(2)} \,  b_{k}^{(1)} \, \chi_{(-\infty,0)}(a_{k}^{(1)}).
\end{equation}
Further, $x_{1}< x_{2}< \ldots < x_{n_{1}}$ are  the ordered values in the set $\{-\frac{b_{k}^{(1)}}{a_{k}^{(1)}}: \, k=1, \ldots , n_{1} \}$ and $\tilde{\mathbf A}^{(2)} = (\tilde{a}^{(2)}_{j,k})_{j,k=1}^{n_{2},n_{1}}$  is obtained by permutation of the columns of 
$(a_{j,k}^{(2)} |a_{k}^{(1)}|)_{j,k=1}^{n_{2},n_{1}}$  according to the ordering of the breakpoints $x_{k}$, $k=1, \ldots , n_{1}$.
In other words, with  $\tilde{\mathbf c}^{(2)} =(\tilde{c}_{j}^{(2)})_{j=1}^{n_{2}}$, $\tilde{\mathbf b}^{(2)} = (\tilde{b}_{j}^{(2)})_{j=1}^{n_{2}}$ and  ${\mathbf x} =(x_{k})_{k=1}^{n_{1}}$ %and $\widetilde{\mathbf A}^{(2)} = (\tilde{a}_{j,k}^{(2)})_{j,k=1}^{n_{2},n_{1}}$ 
we can equivalently rewrite the model (\ref{twol}) as
\begin{equation} 
 f_{3}(t) = (c^{(3)}t + b^{(3)}) + {\mathbf A}^{(3)} \, \sigma \Big( \tilde{\mathbf A}^{(2)}  \sigma ( t {\mathbf 1}  - {\mathbf x}) + \tilde{\mathbf c}^{(2)}t + \tilde{\mathbf b}^{(2)} \Big), \label{twolalt1}
\end{equation}
 where ${\mathbf 1}$ is the vector of ones of length $n_{1}$.
  
2. All functions $f_{2,j}$ in (\ref{f2j}) can be understood  as  the output of a shallow ReLU NN, i.e., $f_{2,j}  \in {\mathcal Y}_{n_{1}}$, and possess the same (possible) breakpoints  $x_{k}$, $k=1, \ldots , n_{1}$. 
%Observe that not all breakpoints need to be active at the same time for all $f_{2,j}(t)$. Here, we say that a breakpoint $x_{k}$ is \textsl{active}, if the term $\sigma(t - x_{k})$ does not vanish  in the representation of the function.
Let $x_{0}:=-\infty$ and $x_{n_{1}+1} := \infty$.  Then  %$f_{2,j}(t)$ are of the form 
\begin{equation}\label{f2form}
f_{2,j}(t) = \mu_{j,\nu} t + \eta_{j,\nu}  \qquad \text{for} \quad t \in (x_{\nu}, x_{\nu+1}), \qquad j=1, \ldots , n_{2}, \; \nu=0, \ldots , n_{1},
\end{equation}
where $\mu_{j,0} := \tilde{c}_{j}^{(2)}$, $\eta_{j,0} :=\tilde{b}_{j}^{(2)}$, and 
\begin{equation} \label{munu}
\mu_{j,\nu}:= \tilde{c}^{(2)}_{j} + \sum\limits_{k=1}^{\nu} \tilde{a}_{j,k}^{(2)} , \qquad \eta_{j,\nu} :=  \tilde{b}^{(2)}_{j} -\sum\limits_{k=1}^{\nu} \tilde{a}_{j,k}^{(2)} x_{k}, \quad  j=1, \ldots , n_{2}, \; \nu=1, \ldots , n_{1}.
\end{equation}
We apply Lemma \ref{lemsigma} to  $f_{2,j}(t)$ for $j=1, \ldots, n_{2}$ and obtain from (\ref{f3mitf2}) 
that $f_{3}$ can be represented in the form 
\begin{align*}
	f_{3}(t) = q_1t+q_0+\sum_{k=1}^{n_1}\alpha_k\sigma(t-x_k)+\sum_{j=1}^{n_2}\sum_{\nu=0}^{n_1}\alpha_{j,\nu}\sigma(t-x_{j,\nu}),
\end{align*}
where $x_{1} < x_{2} < \ldots < x_{n_{1}}$ are the breakpoints of $f_{2,j}$,
$$
 x_{j,\nu}:=  \left\{ \begin{array}{ll}
-\frac{\eta_{j,\nu}}{\mu_{j,\nu}}  & \textrm{for} \, \mu_{j,\nu} \neq 0 , \\
-\infty & \textrm{for} \, \mu_{j,\nu} = 0, \end{array} \right. \qquad j=1, \ldots  ,n_2, \, \nu=0, \ldots , n_1, $$
with $\mu_{j,\nu}$ and $\eta_{j,\nu}$ defined in (\ref{munu})
and with 
\begin{align} \label{q1}
q_1 &:= c^{(3)}-\sum_{j=1}^{n_{2}} a_{j}^{(3)} \sigma(-\tilde{c}_j^{(2)}),\\
\label{q0}
	q_0 &:=  b^{(3)}+\sum_{j=1}^{n_{2}} a_{j}^{(3)} \Big(\tilde{b}_j^{(2)} \chi_{(-\infty,0)}(\tilde{c}_j^{(2)})+\sigma(\tilde{b}_j^{(2)})\chi_{\{0\}}(\tilde{c}_j^{(2)}) \Big),\\
	\label{alphak}
	\alpha_k &:= \sum_{j=1}^{n_{2}} a_{j}^{(3)} \Big(\tilde{a}_{j,k}^{(2)} \, \chi_{(0,\infty)}(f_{2,j}(x_{k})) + \chi_{\{0\}}(f_{2,j}(x_{k}))(\sigma(\mu_{j,k})+\sigma(-\mu_{j,k-1})) \Big),\\
	\label{alphanuk}
	\alpha_{j,\nu} &:= a_j^{(3)} |\mu_{j,\nu}| \chi_{(x_\nu,x_{\nu+1})}(x_{j,\nu}).
\end{align}
In particular, $x_{j,\nu}= - \infty$ is not an active breakpoint since we have $\alpha_{j,\nu}=0$ in this case.
\end{proof}

The representation of $f_{3}$ in Theorem \ref{theo1} implies

\begin{corollary}\label{cor1}
Any function $f_{3} \in {\mathcal Y}_{n_{1},n_{2}}$ in $(\ref{twol})$ is a piecewise continuous spline function with at most $n_{1}n_{2} + n_{1}+n_{2}$ breakpoints, i.e., all functions $f_3 \in {\mathcal Y}_{n_1,n_2}$ are also contained in $ \Sigma_{N}$ for $N=(n_{1}+1)(n_{2}+1)-1$.
\end{corollary}

Moreover, the proof of Theorem \ref{theo1} implies, how we can transfer the model ${\mathcal Y}_{n_1,n_2}$ into the model $\Sigma_N$, see Algorithm \ref{algo2} in Section \ref{sec:more}.

\begin{corollary}\label{corscaling}
The \textnormal{ReLU DNN} model ${\mathcal Y}_{n_{1},n_{2}}$ in $(\ref{DNN})$ can be equivalently described in the form 
\begin{equation}
f_{3}(t) = (c^{(3)}t + b^{(3)}) + {\mathbf A}^{(3)} \, \sigma \Big({\mathbf A}^{(2)}  \sigma ( t {\mathbf 1}  - {\mathbf x}) + {\mathbf c}^{(2)}t + {\mathbf b}^{(2)} \Big), \label{twolalt}
\end{equation}
depending on the parameter set $\{{\mathbf A}^{(3)}, {\mathbf A}^{(2)}, \, {b}^{(3)}, \, {\mathbf b}^{(2)}, \, {\mathbf x}, \, c^{(3)}, \, {\mathbf c}^{(2)} \}$ where ${\mathbf A}^{(3)} \in {\mathbb R}^{1 \times n_{2}}$, ${\mathbf A}^{(2)} \in {\mathbb R}^{n_{2} \times n_{1}}$, ${\mathbf b}^{(2)} \in {\mathbb R}^{n_{2}}$, ${\mathbf x}\in {\mathbb R}^{n_{1}}$ with $x_{1}<x_{2}< \ldots < x_{n_{1}}$, $b^{(3)}, c^{(3)} \in {\mathbb R}$, and  ${\mathbf c}^{(2)} = \sign({\mathbf c}^{(2)} ) \in \{0,-1,1\}^{n_{2}}$. Thus, ${\mathcal Y}_{n_{1},n_{2}}$ depends on at most $n_{1}n_{2}+2n_{2}+n_{1}+2$ real parameters and $n_{2}$ sign parameters.
\end{corollary}

\begin{proof}
The representation of $f_{3}$ with ${\mathbf c}^{(2)} \in {\mathbb R}^{n_{2}}$ follows already from the proof of Theorem \ref{theo1}. Now, we apply the positive scaling property in Remark \ref{rem11} with ${\mathbf D}= \diag(d_{j})_{j=1}^{n_{2}}$
 with entries $d_{j} = |{c}_{j}^{(2)}|$ for $|{c}_{j}^{(2)}| >0$ and $d_{j} = 1$ for ${c}_{j}^{(2)}=0$, and obtain 
\begin{align*}
 f_{3}(t)  &= ({c}^{(3)}t + {b}^{(3)}) + {\mathbf A}^{(3)} \, \sigma \Big( {\mathbf A}^{(2)}  \sigma ({\mathbf 1} t - {\mathbf x}) + {\mathbf D} \, \sign({\mathbf c}^{(2)})t + {\mathbf b}^{(2)} \Big) \\
 &=  (c^{(3)}t + b^{(3)}) + {\mathbf A}^{(3)} \, \sigma \Big( {\mathbf D} \Big( {\mathbf D}^{-1} {\mathbf A}^{(2)}  \sigma ({\mathbf 1} t - {\mathbf x}) +  \sign({\mathbf c}^{(2)})t + {\mathbf D}^{-1}{\mathbf b}^{(2)} \Big) \Big) \\
 &= (c^{(3)}t + b^{(3)}) + {\mathbf A}^{(3)} {\mathbf D} \, \sigma   \Big( {\mathbf D}^{-1} {\mathbf A}^{(2)}  \sigma ({\mathbf 1} t - {\mathbf x}) +  \sign({\mathbf c}^{(2)})t + {\mathbf D}^{-1}{\mathbf b}^{(2)} \Big) .
\end{align*} 
Thus we find (\ref{twolalt}) if we replace ${\mathbf A}^{(3)} {\mathbf D}$ by ${\mathbf A}^{(3)}$, 
${\mathbf D}^{-1} {\mathbf A}^{(2)}$ by ${\mathbf A}^{(2)}$, ${\mathbf D}^{-1}{\mathbf b}^{(2)}$ by ${\mathbf b}^{(2)}$, and set ${\mathbf c}^{(2)}= \sign {\mathbf c}^{(2)}$. 
Hence, the model ${\mathcal Y}_{n_{1}, n_{2}}$ depends on at most $n_{1}n_{2} + 2n_{2} + n_{1} + 2$ real parameters and one sign vector of length $n_{2}$.
\end{proof}

We will show in Subsection \ref{secred} that  the parameter set determining the model  ${\mathcal Y}_{n_{1},n_{2}}$ in Corollary \ref{corscaling} is a set of  independent parameters. 

\begin{remark}
Theorem \ref{theo1} implies that the set of possible breakpoints of $f_{3}(t) = c^{(3)}t + b^{(3)} + \sum\limits_{j=1}^{n_{2}} a_{j}^{(3)} \, \sigma(f_{2,j}(t))$ in (\ref{f3mitf2}) is composed  of the set of  breakpoints $x_{1}, \ldots , x_{n_{1}}$ of $f_{2,j}$ (first level breakpoints) and the set of zeros of $f_{2,j}$, namely $x_{j,\nu}$, $j=1, \ldots , n_{2}$, in each interval $(x_{\nu}, x_{\nu+1})$, $\nu=0, \ldots, n_{1}$ (second level breakpoints).
\end{remark}

\begin{example}\label{ex1}
We consider the function $f_{3}$ in (\ref{twol}) with $n_1=n_2=3$,
$$ {\mathbf A}^{(1)} = \begin{pmatrix}
1 \\ -1 \\ -1 \end{pmatrix}, \quad {\mathbf b}^{(1)} = \begin{pmatrix} 
-1 \\ 2 \\ 3 \end{pmatrix}),\quad {\mathbf A}^{(2)} = \begin{pmatrix}  -2 & 2 & -3 \\
-1 & 1 & -1.5 \\
1 & -2 & 2.5  \end{pmatrix}, \quad {\mathbf b}^{(2)} = \begin{pmatrix}
4.5 \\ 2.2\\ -3.3 \end{pmatrix},
$$
${\mathbf A}^{(3)}= (1,1,1)$, ${\mathbf c}^{(2)} ={\mathbf 0}$, and  $c^{(3)}= b^{(3)}=0$.
Observe that in this example no source channel is used since ${\mathbf c}^{(2)}$ and $c^{(3)}$ vanish.
We can rewrite $f_{3}$ as in (\ref{twolalt1}) with 
$\tilde{\mathbf A}^{(2)} = {\mathbf A}^{(2)}$, ${\mathbf x} = (1,2,3)^{T}$, $\tilde{\mathbf b}^{(2)} = (-0.5, -0.3, 0.2)^{T}$ and 
$\tilde{\mathbf c}^{(2)} = ( 1, 0.5, -0.5)^{T}$. 
The function $f_{3}$ possesses the maximal number of $n_1n_2+n_1+n_2=9+6=15$ breakpoints
\begin{align*}
x_{1} &= 1, \qquad x_{2} = 2, \qquad x_{3} = 3, \\
x_{1,0} &= 0.5, \qquad x_{1,1} = 1.5, \qquad x_{1,2} = 2.5, \qquad  x_{1,3} = 3.25, \\
x_{2,0} &= 0.6,  \qquad x_{2,1} = 1.4, \qquad x_{2,2} = 2.6, \qquad  x_{2,3} = 3.2, \\
x_{3,0} &= 0.4, \qquad x_{3,1} = 1.6, \qquad x_{3,2} = 2.1\overline{3}, \qquad  x_{3,3} = 4.3. 
\end{align*}
Further, we obtain the parameters $\mu_{j,\nu}$ in (\ref{munu}), 
\begin{align*}
\mu_{1,0} &= 1, \qquad \mu_{1,1} = -1, \qquad \mu_{1,2} = 1, \qquad  \mu_{1,3} = -2, \\
\mu_{2,0} &= 0.5,  \qquad \mu_{2,1} = -0.5, \qquad \mu_{2,2} = 0.5, \qquad  \mu_{2,3} = -1, \\
\mu_{3,0} &= -0.5 \qquad \mu_{3,1} = 0.5, \qquad \mu_{3,2} = -1.5, \qquad  \mu_{3,3} = 1, 
\end{align*}
and  the coefficients in the representation (\ref{mo2})  are of the form
\begin{align*}
\alpha_{1} &= -3, \qquad \alpha_{2} = -2, \qquad \alpha_{3} = -4.5. \\
\alpha_{1,0} &= 1, \qquad \alpha_{1,1} = 1, \qquad \alpha_{1,2} = 1, \qquad  \alpha_{1,3} = 2, \\
\alpha_{2,0} &= 0.5,  \qquad \alpha_{2,1} = 0.5, \qquad \alpha_{2,2} = 0.5, \qquad  \alpha_{2,3} = 1, \\
\alpha_{3,0} &= 0.5, \qquad \alpha_{3,1} = 0.5, \qquad \alpha_{3,2} = 1.5, \qquad  \alpha_{3,3} = 1.
\end{align*}
The spline function $f_{3}$ is illustrated in Figure \ref{fig1}, where the green dots mark the first level knots $x_{k}$, $k=1,2,3$, and the orange stars mark the second level knots $x_{j,\nu}$, $j=1, 2,3$, $\nu=0, 1,2,3$.
\begin{figure}[h]
\begin{center}
	\includegraphics[scale=0.35]{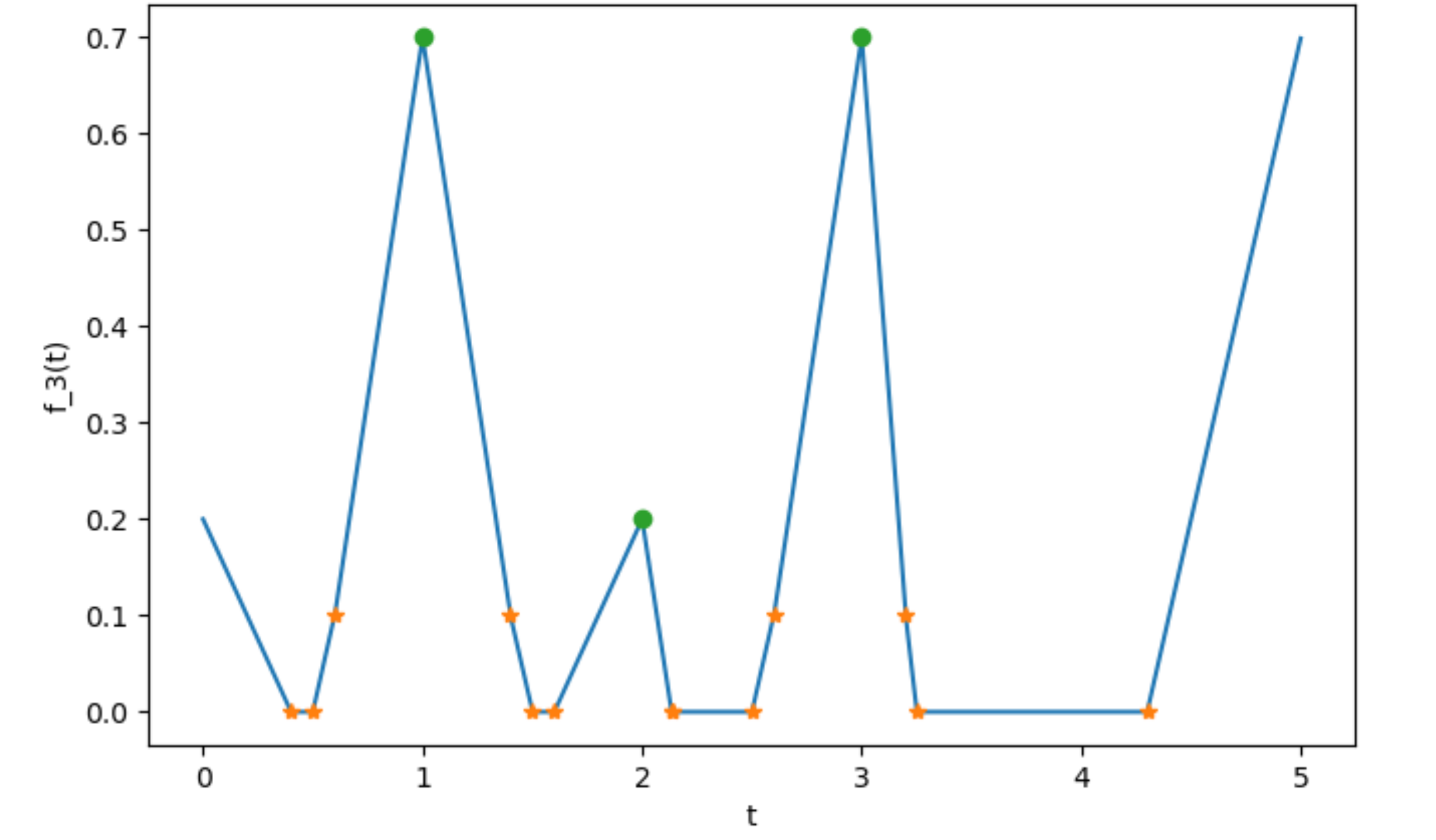}
\end{center}
\caption{Illustration of the function $f_{3}$ in Example \ref{ex1}.}
\label{fig1}
\end{figure}
\end{example}

\begin{remark}
Similarly as in Corollary \ref{corscaling}, we can apply a positive rescaling to obtain a representation of  $f_{3}$ in Example \ref{ex1} with $\tilde{\mathbf c}^{(2)}= \sign (\tilde{\mathbf c}^{(2)})$. Let ${\mathbf D} =  \diag (|\tilde{c}_{j}^{(2)})_{j=1}^{3} = \diag(1, \frac{1}{2},\frac{1}{2})$ and set 
$$ \tilde{\mathbf A}^{(2)} = {\mathbf D}^{-1}
\begin{pmatrix}  -2 & 2 & -3 \\
-1 & 1 & -1.5 \\
1 & -2 & 2.5  \end{pmatrix}, \quad  \tilde{\mathbf b}^{(2)}={\mathbf D}^{-1} \begin{pmatrix} 
-0.5 \\ -0.3\\ 0.2 \end{pmatrix}, \qquad {\mathbf A}^{(3)} = (1,1,1) \, {\mathbf D}^{-1}. $$
Then $f_{3}(t) = {\mathbf A}^{(3)} \sigma\Big( \tilde{\mathbf A}^{(2)} \sigma( t{\mathbf 1} -{\mathbf x}) + t \tilde{\mathbf c}^{(2)} + \tilde{\mathbf b}^{(2)} \Big)$
with a sign vector  $\tilde{\mathbf c}^{(2)} = (1,1,-1)^{T}$.
\end{remark}

\subsection{Two hidden layer ReLU NN with maximal number of breakpoints}
\label{sec:break}

In this subsection, we will investigate the structure of  functions $f_3 \in {\mathcal Y}_{n_1, n_2}$ that possess the maximal number of $(n_{1}+1 )(n_{2} +1) -1$ breakpoints. Moreover, we will give a procedure, how to construct such ReLU DNN. 
These investigations will be also crucial  to determine the number of independent parameters in the model ${\mathcal Y}_{n_1, n_2}$.

Our observations in the proof of Theorem \ref{theo1}  show that 
the maximal number of  breakpoints can only be achieved if all breakpoints $x_{k}$, $k=1, \ldots , n_{1}$, and $x_{j,\nu}$, $j=1, \ldots, n_{2}$, $\nu=0, \ldots , n_{1}$, in the representation of $f_{3}$ in  (\ref{mo2}) are active, i.e., if all coefficients $\alpha_{\ell}$, $\ell=1, \ldots , n_{1}$ and $\alpha_{j,\nu}$, $j=1, \ldots, n_{2}$, $\nu=0, \ldots , n_{1}$, in  (\ref{alphak})--(\ref{alphanuk}) are nonzero.
In this subsection, we will use the model
\begin{align}\label{twol2}
 f_{3}(t) &= (c^{(3)}t + b^{(3)}) + {\mathbf A}^{(3)} \, \sigma \Big( {\mathbf A}^{(2)}  \sigma ( t {\mathbf 1}  - {\mathbf x}) + {\mathbf c}^{(2)}t + {\mathbf b}^{(2)} \Big),\\
 \label{twof2}
 &= (c^{(3)}t + b^{(3)}) + \sum_{j=1}^{n_{2}} a_{j}^{(3)} \, \sigma \Big( f_{2,j}(t) \Big)
\end{align}
with ${\mathbf c}^{(2)} = \sign ({\mathbf c}^{(2)})$, i.e., ${\mathbf c}^{(2)}  \in \{-1,0,1\}^{n_{2}}$
which is equivalent to (\ref{twol}),  as shown in Corollary \ref{corscaling}. Here $f_{2,j}(t)$ is assumed to be in the spline model form 
$$ f_{2,j}(t)= {c}_{j}^{(2)}t + {b}_{j}^{(2)} + \sum_{k=1}^{n_{1}} a_{j,k}^{(2)}  \sigma ( t   - x_{k}), \qquad j=1, \ldots , n_{2}. $$  
Further, we recall that in this notation $f_{2,j}(t)= \mu_{j,\nu} t + \eta_{j,\nu}$ for $t \in (x_{\nu}, x_{\nu+1})$ with 
$\mu_{j,0}=c_{j}^{(2)}$, $\eta_{j,0}= b_{j}^{(2)}$ and 
\begin{equation}\label{munu1}
\mu_{j,\nu} = c_{j}^{(2)} + \sum_{k=1}^{\nu} a_{j,k}^{(2)}, \qquad \eta_{j,\nu}= b_{j}^{(2)} - \sum_{k=1}^{\nu} a_{j,k}^{(2)} x_{k}, \quad j=1, \ldots , n_{2}, \; \nu=1, \ldots , n_{1}. 
\end{equation}

We will show
\begin{theorem}\label{theomax}
Let $ n_1 \ge 1$ and $n_{2} \ge 2$.
Then for the maximal number of $(n_{1}+1)(n_{2}+1)-1$ 
prescribed pairwise distinct  breakpoints on ${\mathbb R}$ there exists a function $f_3 \in {\mathcal Y}_{n_{1},n_{2}}$ 
that possesses these breakpoints and they are all active.
\end{theorem}

To prove Theorem \ref{theomax}, we will derive a procedure to construct  $f_3 \in {\mathcal Y}_{n_{1},n_{2}}$  with $(n_1+1)(n_2+1)-1$ prescribed breakpoints.
We start with considering  in more detail, how the parameter sets $\{ {\mathbf A}^{(3)}, \, {\mathbf A}^{(2)}, {\mathbf x}, \, {\mathbf c}^{(2)}, \, {\mathbf b}^{(2)}\}$ 
determining $f_{3}$ in (\ref{twol2}) (up to a linear polynomial) and 
\begin{equation}\label{splinepar}
\{ (x_{k})_{k=1}^{n_{1}}, (x_{j,\nu})_{j=1,\nu=0}^{n_{2}, n_{1}}, \, ({\alpha}_{\ell})_{\ell=1}^{n_{1}}, \, (\alpha_{j,\nu})_{j=1,\nu=0}^{n_{2},n_{1}} \}
\end{equation}
determining the CPL spline representation of ${f}_{3}$ in (\ref{mo2}) (up to a linear polynomial) are related and which redundancies appear.

The next lemma shows, how the slopes $\partial_{+}f_{2,j}(x_{\nu}) =\mu_{j,\nu}$ of $f_{2,j}$ in $(x_{\nu}, x_{\nu+1})$ depend on the breakpoints of $f_{3}$. 

\begin{lemma}\label{lem2}
For a given function $f_3 \in {\mathcal Y}_{n_{1},n_{2}}$ in $(\ref{twof2})$ with $n_{1}, n_{2} \ge 1$  let $\mu_{j,\ell}, \, \eta_{j,\ell}$ be given as in $(\ref{munu1})$. Assume that   
$x_{j,\ell} := - \frac{\eta_{j,\ell}}{\mu_{j,\ell}}$ are well-defined, i.e., $\mu_{j,\ell} \neq 0$, for $j=1, \ldots , n_2, \, \ell=0, \ldots , n_1$. Then 
\begin{equation}\label{murec} \textstyle \mu_{j,\nu} = \mu_{j,\nu-1} \, \left( \frac{x_{\nu} - x_{j,\nu-1}}{x_{\nu}- x_{j,\nu}} \right)
\end{equation}
for all $j=1, \ldots , n_2, \, \nu=1, \ldots , n_1$ with $x_\nu \neq x_{j,\nu}$. 
If additionally $x_{j,\nu} \in (x_{\nu}, x_{\nu+1})$ for $j=1, \ldots , n_2, \, \nu=0, \ldots , n_1$ (with $x_{0}:=-\infty$ and $x_{n_{1}+1} := \infty$), then
\begin{equation}\label{rec1}
\mu_{j,0} = c_{j}^{(2)}, \qquad  \mu_{j,\nu} =  \textstyle {c}_{j}^{(2)} \, \prod\limits_{\ell=1}^{\nu} \left( \frac{x_{\ell} - x_{j,\ell-1}}{x_{\ell} - x_{j,\ell}} \right), \qquad {j}=1, \ldots, n_{2}, \, \nu=1, \ldots , n_{1}.
 \end{equation}
In particular, $c_{j}^{(2)}= \sign(c_{j}^{(2)}) \neq 0$ for $j=1, \ldots , n_{2}$, and  $\sign(\mu_{j,\nu}) = (-1)^{\nu} \sign({c}_{j}^{(2)})$, i.e., the functions $f_{2,j}$ in $(\ref{f2form})$ have alternating slopes.
\end{lemma}

\begin{proof}
Observe from (\ref{munu1}) and the continuity of $f_{2,j}$ that 
$$ f_{2,j}(x_{\nu}) = \mu_{j,\nu} x_{\nu} + \eta_{j,\nu} = \mu_{j,\nu-1}x_{\nu} + \eta_{j,\nu-1}, \qquad j=1, \ldots, n_{2}, \, \nu=1, \ldots , n_{1}. $$
Using that $x_{j,\nu} = -\frac{\eta_{j,\nu}}{\mu_{j,\nu}}$, we obtain
$$\mu_{j,\nu} x_{\nu} - \mu_{j,\nu} x_{j,\nu} = \mu_{j,\nu-1} x_{\nu} - \mu_{j,\nu-1} x_{j, \nu-1}. $$
For $x_{\nu} \neq x_{j,\nu}$ we conclude the recursion (\ref{murec}).
For  $x_{j,\nu} \in (x_{\nu}, x_{\nu+1})$ we have $x_{\nu}- x_{j,\nu} < 0$, $x_{\nu} - x_{j, \nu-1} >0$.  We can therefore apply the recursion iteratively, where each term $\left( \frac{x_{\ell} - x_{j,\ell-1}}{x_{\ell} - x_{j,\ell}} \right)$ in the product (\ref{rec1}) ist negative.
This yields the assertion, since $\mu_{j,0} = {c}_{j}^{(2)}$ by (\ref{munu1}).
\end{proof}

\begin{remark}
The function $f_{3}$ can only possess the maximal number of breakpoints, if $f_{2,j} (x_{\nu})  \neq 0$ with alternating sign change, i.e., $\sign(f_{2,j}(x_{\nu})) = (-1)^{\nu-1}\sign(f_{2,j}(x_{1}))$ for $\nu=2, \ldots, n_{1}$.
Then, the slopes $\mu_{j,\nu}$ of $f_{2,j}$ in $(x_{\nu},x_{\nu+1})$ also change their sign, such that the second layer breakpoints $x_{j,\nu}$ satisfy $f_{2,j}(x_{j,\nu}) = \mu_{j,\nu} x_{j,\nu} + \eta_{j,\nu} =0$.
\end{remark}

\begin{corollary} \label{lem3}
For a given function $f_3 \in {\mathcal Y}_{n_{1},n_{2}}$ in $(\ref{twof2})$  with maximal number of $(n_{1}+1)(n_{2}+1
)-1$ breakpoints, $x_{1}< x_{2} < \ldots  < x_{n_{1}}$ and $x_{j,\nu} \in (x_{\nu}, \, x_{\nu+1})$ (with the notation $x_{0}:= -\infty$ and $x_{n_{1}+1} := \infty$), the parameter matrix ${\mathbf A}^{(2)} = ({a}_{j,k}^{{(2)}})_{j,k=1}^{n_{2}, n_{1}}$ is already determined  by
\begin{equation}\label{ajk2} \textstyle {a}_{j,k}^{{(2)}} = \mu_{j,k} - \mu_{j,k-1} = {c}_{j}^{(2)} \left( \frac{x_{j,k}-x_{j,k-1}}{x_{k} - x_{j,k}} \right) \prod\limits_{\ell=1}^{{k-1}} \left(\frac{x_{\ell} - x_{j,\ell-1}}{x_{\ell} - x_{j,\ell}} \right). 
\end{equation}
In particular, all components $a_{j,k}^{(2)}$ do not vanish and have alternating sign,
$$ \sign({a}_{j,k}^{(2)}) = \sign(\mu_{j,k}) = (-1)^{k} {c}_{j}^{(2)}, \qquad {j=1, \ldots, n_{2}, \, k=1, \ldots , n_{1}.}$$

\end{corollary}

\begin{proof} The representation for ${a}_{j,k}^{{(2)}}$ follows directly from (\ref{munu1}) and Lemma \ref{lem2}.
The alternating sign change follows since the $k$ factors in the representation of ${a}_{j,k}^{(2)}$ are all negative. 
\end{proof}

\begin{corollary}\label{cor3}
Let a given function $f_3 \in {\mathcal Y}_{n_{1},n_{2}}$ in $(\ref{twol2})$ (with ${\mathbf c}^{(2)} = \sign({\mathbf c}^{(2)}) \in \{-1,1\}^{n_{2}}$ ) be represented in the form $(\ref{mo2})$ and 
assume that $f_{3}$ possesses the maximal number of $(n_{1}+1)(n_{2}+1)-1$ active breakpoints $(x_{k})_{k=1}^{n_{1}}$ and $(x_{j,\nu})_{j=1,\nu=0}^{n_{2}, n_{1}}$. Then the coefficients of $f_3$ in the representation $(\ref{mo2})$ satisfy $\alpha_{j,0} = a_{j}^{(3)}$, 
\begin{align*}
\alpha_{j,\nu} &=  \textstyle a_{j}^{(3)}   \prod\limits_{\ell=1}^{\nu} \left| \frac{x_{\ell} - x_{j,\ell-1}}{x_{\ell} - x_{j,\ell}} \right|, \quad j=1, \ldots , n_{2}, \, \nu=1, \ldots , n_{1},\\
\alpha_{\ell} &= \textstyle \frac{1}{2} \sum\limits_{j=1}^{{n_2}} a_{j}^{(3)}  \left( \frac{x_{j,\ell}-x_{j,\ell-1}}{x_{\ell} - x_{j,\ell}} \right) \prod\limits_{k=1}^{{\ell-1}} \left(\frac{x_{k} - x_{j,k-1}}{x_{k} - x_{j,k}} \right) \, \big( (-1)^{\ell+1} + \sign({c}_{j}^{(2)}) \big),
\end{align*}
i.e., the  coefficients $\alpha_{j,\nu}$ and $\alpha_{\ell}$ are already determined by the knot sets $(x_{j,\nu})_{j=1,\nu=0}^{n_{2}, n_{1}}$, ${\mathbf x} = (x_{\ell})_{\ell=1}^{n_{1}}$, and the parameter vector ${\mathbf a}^{(3)} =(a_{j}^{(3)})_{j=1}^{n_{2}}$ has only nonzero components, and  ${\mathbf c}^{(2)} = ({c}_{j}^{(2)})_{j=1}^{n_{2}} \in \{-1,1\}^{n_{2}}$ contains changing signs for $n_{1}>1$, i.e., $\sign({c}_{j_{1}}^{(2)}) \neq \sign({c}_{j_{2}}^{(2)})$ for some $j_{1} \ne j_{2}$.
\end{corollary}

\begin{proof}
Formula (\ref{alphanuk}) implies $\alpha_{j,\nu}= a_{j}^{(3)} |\mu_{j,\nu}| \neq 0$ and we use (\ref{rec1}) to obtain the equations for $\alpha_{j,0}$ and $\alpha_{j,\nu}$, $j=1, \ldots , n_{2}, \, \nu=1, \ldots , n_{1}$.

\noindent
To show the relation for $\alpha_{\ell}$, we observe by  Lemma \ref{lem2} and Corollary \ref{lem3} with $x_{j,\ell}=-\frac{\eta_{j,\ell}}{\mu_{j,\ell}} \in (x_{\ell}, x_{\ell+1})$ that
\begin{align*}
\sign(f_{2,j} (x_{\ell})) &= \sign(\mu_{j,\ell} x_{\ell} + \eta_{j,\ell}) = \sign(\mu_{j,\ell} x_{\ell} - \mu_{j,\ell} x_{j,\ell}) 
= \sign(\mu_{j,\ell}) \sign(x_{\ell} - x_{j,\ell}) \\
 &= - \sign(\mu_{j,\ell}) = (-1)^{\ell+1} \sign(\mu_{j,0}) = (-1)^{\ell+1} \sign({c}_{j}^{(2)}) \neq 0,
\end{align*}
i.e., $\chi_{(0,\infty)}(f_{2,j} (x_{\ell})) = \frac{1}{2} \big( (-1)^{\ell+1} \sign({c}_{j}^{(2)}) +1 \big)$. Thus, we obtain by (\ref{alphak}) and (\ref{ajk2}) that 
\begin{align*}
\alpha_{\ell} &= \textstyle \sum\limits_{j=1}^{{n_2} } a_{j}^{(3)} {a}^{(2)}_{j,\ell}  \, \chi_{(0,\infty)}(f_{2,j} (x_{\ell})) \\
%&= \textstyle \frac{1}{2} \sum\limits_{j=1}^{{n_2}} \frac{\alpha_{j,0}}{|\mu_{j,0}|} (\mu_{j,\ell} - \mu_{j,\ell-1}) \, \big( (-1)^{\ell+1} \sign(\mu_{j,0}) +1 \big)\\
%&= \textstyle \frac{1}{2} \sum\limits_{j=1}^{{n_2}} \frac{\alpha_{j,0}}{|\mu_{j,0}|} \mu_{j,0} \, \left( \frac{x_{\ell} - x_{j,\ell-1}}{x_{\ell} - x_{j,\ell}} -1\right) 
%\prod\limits_{k=1}^{\ell-1} \left( \frac{x_{k} - x_{j,k-1}}{x_{k} - x_{j,k}} \right) \big( (-1)^{\ell+1} \sign(\mu_{j,0}) +1 \big)\\
&= \textstyle \frac{1}{2} \sum\limits_{j=1}^{{n_2}} a_{j}^{(3)} \, |{c}_{j}^{(2)}| \left( \frac{x_{j,\ell}-x_{j,\ell-1}}{x_{\ell} - x_{j,\ell}} \right) \prod\limits_{k=1}^{{\ell-1}} \left(\frac{x_{k} - x_{j,k-1}}{x_{k} - x_{j,k}} \right) \, \big((-1)^{\ell+1}+\sign({c}_{j}^{(2)}) \big).
\end{align*}
In particular, if ${c}_{j}^{(2)}$ had the same sign for all $j=1, \ldots , n_{2}$, then we would find either $\alpha_{2\ell}=0$, if $\sign({c}_{j}^{(2)}) =1$ for all $j$, or  $\alpha_{2\ell-1}=0$, if $\sign({c}_{j}^{(2)}) =-1$ for all $j$. In both cases the first sum in (\ref{mo2}) would degenerate. Therefore, for $n_{1}>1$ the maximal number of breakpoints of $f_{3}$ can only by achieved, if there exist $j_{1}, j_{2}$ with $\sign({c}_{j_{1}}^{(2)}) \neq \sign({c}_{j_{2}}^{(2)})$. In particular, for $n_{1}>1$, we need to have $n_{2}\ge 2$ to achieve the maximal number of breakpoints.
\end{proof}

With these preliminaries, we are now ready to prove the main theorem of this section that covers the assertions of Theorem \ref{theomax}.
\begin{theorem}\label{theomax1}
For integers $ n_1 \ge 1$ and $n_{2} \ge 2$ let $(n_1+1)(n_2+1)-1$ real pairwise distinct knots be given, which are ordered  such that 
$$  -\infty< x_{1} < x_{2} < \ldots < x_{n_{1}} < \infty  $$
and
$$ x_{j,0} \in (-\infty , x_{1}), \; x_{j,n_{1}}  \in (x_{n_{1}}, \infty), \quad x_{j,\nu} \in (x_{\nu}, x_{\nu+1}), \quad \nu=1, \ldots , n_{1}-1, \; j=1, \ldots , n_{2}. $$
Then $f_{3} \in {\mathcal Y}_{n_{1},n_{2}}$ 
in $(\ref{twol2})$  determined by  
\begin{align*} 
{\mathbf x} &:= (x_{k})_{k=1}^{n_{1}}, \qquad {\mathbf c}^{(2)} := ((-1)^{j+1})_{j=1}^{n_{2}}, \qquad {\mathbf b}^{(2)} := (- x_{j,0}\, {c}_{j}^{(2)})_{j=1}^{n_{2}} = (x_{j,0}\, (-1)^{j})_{j=1}^{n_{2}}, \\
\mu_{j,0} &:= c_{j}^{(2)} = (-1)^{j+1}, \quad \mu_{j,\nu} := \textstyle (-1)^{j+1} \, \prod\limits_{\ell=1}^{\nu} \left( \frac{x_{\ell} - x_{j,\ell-1}}{x_{\ell} - x_{j,\ell}} \right), \quad {j}=1, \ldots, n_{2}, \, \nu=1, \ldots , n_{1}, \\
{\mathbf A}^{(2)} &:= ({a}_{j,k}^{(2)})_{j,k=1}^{n_{2}, n_{1}} \quad \text{with} \quad {a}_{j,k}^{{(2)}} = \mu_{j,k} - \mu_{j,k-1}, \\
{\mathbf A}^{(3)} &:=((a_{j}^{(3)})_{j=1}^{n_{2}})^{T}  \in {\mathbb R}^{1 \times n_{2}} \quad \text{with} \quad  \sign({\mathbf A}^{(3)}) = ((\pm (-1)^{j})_{j=1}^{n_{2}})^{T}, \, c^{(3)}, \, b^{(3)} \in {\mathbb R},
\end{align*}
possesses the $(n_{1}+1)(n_{2}+1)-1$ prescribed knots as active  breakpoints. More precisely, $f_{3}$ is of the form
\begin{equation}\label{f3spline} f_{3}(t) = {q}_{1}t + q_{0} + \sum_{\ell=1}^{n_{1}} \alpha_{\ell} \, \sigma(t-x_{\ell}) + \sum_{j=1}^{n_{2}}\sum_{\nu=0}^{n_{1}} \alpha_{j,\nu} \, \sigma(t-x_{j,\nu})
\end{equation}
with $\alpha_{j,0} = a_{j}^{(3)} \neq 0$, $j=1, \ldots, n_{2}$, 
\begin{align*} \alpha_{j,\nu} &=  \textstyle a_{j}^{(3)} \prod\limits_{\ell=1}^{\nu} \left| \frac{x_{\ell} - x_{j,\ell-1}}{x_{\ell} - x_{j,\ell}} \right| \neq 0, \quad j=1, \ldots , n_{2}, \, \nu=1, \ldots , n_{1}, \\
\alpha_{\ell} &= \textstyle (-1)^{\ell+1} \textstyle\sum\limits_{\substack{j=1\\ j+\ell \, \text{even}}}^{n_{2}} a_{j}^{(3)} \left( \frac{x_{j,\ell}-x_{j,\ell-1}}{x_{\ell} - x_{j,\ell}} \right) \prod\limits_{k=1}^{{\ell-1}} \left(\frac{x_{k} - x_{j,k-1}}{x_{k} - x_{j,k}} \right) \neq 0 , \quad \ell=1, \ldots , n_{1},\\
q_{1} &= c^{(3)} - \sum_{\substack{j=1\\ j \, \text{even}}}^{n_{2}} a_{j}^{(3)}, \qquad q_{0} = b^{(3)} + \sum_{\substack{j=1\\ j \, \text{even}}}^{n_{2}} a_{j}^{(3)}\, x_{j,0}.
\end{align*}
%and with arbitrary $q_{1}, q_{0} \in {\mathbb R}$.
\end{theorem}

\begin{proof} 
We show that the function $f_{3}$ in (\ref{twol2}) determined by the parameters given in the theorem  possesses the representation (\ref{f3spline}). The structure of $f_{3}$ in (\ref{twol2}) already implies that $f_{3}$ possesses the first layer breakpoints $(x_{\ell})_{\ell=1}^{n_{1}}$.
The second layer breakpoints are obtained by $-\frac{\eta_{j,\nu}}{\mu_{j,\nu}}$ with $\mu_{j,\nu}$ as given in the theorem and with 
$\eta_{j,\nu}$ in (\ref{munu1}). Obviously, we obtain 
$$ \textstyle -\frac{\eta_{j,0}}{\mu_{j,0}}= -\frac{b_{j}^{(2)}}{c_{j}^{(2)}}=x_{j,0}, \qquad j=1, \ldots , n_{2}. $$ 
Further, 
using the recursions $\mu_{j,\nu} =\mu_{j,\nu-1} \left( \frac{x_{\nu} - x_{j,\nu-1}}{x_{\nu} - x_{j,\nu}} \right)$ and 
$\eta_{j,\nu} = \eta_{j,\nu-1} -a_{j,\nu}^{(2)} x_{\nu} = \eta_{j,\nu-1}-(\mu_{j,\nu} - \mu_{j,\nu-1}) x_{\nu}$ we obtain inductively from $\eta_{j,\nu-1}= -\mu_{j,\nu-1} x_{j,\nu-1}$ that 
\begin{align*} \textstyle -\frac{\eta_{j,\nu}}{\mu_{j,\nu}} &= \textstyle - \frac{-x_{j,\nu-1}\mu_{j,\nu-1} -(\mu_{j,\nu} - \mu_{j,\nu-1}) x_{\nu}}{\mu_{j,\nu}} = 
- \frac{\mu_{j,\nu-1}(x_{\nu}-x_{j,\nu-1}) - \mu_{j,\nu}x_{\nu}}{\mu_{j,\nu}} \\
&= \textstyle - \frac{\mu_{j,\nu}(x_{\nu}- x_{j,\nu}) -  \mu_{j,\nu}x_{\nu}}{\mu_{j,\nu}} = x_{j,\nu}, \quad j=1, \ldots , n_{2}, \; \nu=1, \ldots , n_{1}, 
\end{align*}
i.e., $f_{3}$ possesses the breakpoints $x_{j,\nu}$ as given in (\ref{f3spline}). 

With  ${\mathbf c}^{(2)}$  and ${\mathbf A}^{(3)}$ and the fixed breakpoints $(x_{\ell})_{\ell=1}^{n_{1}}$ and $(x_{j,\nu})_{j=1,\nu=0}^{n_{2}, n_{1}}$ as given in Theorem \ref{theomax1} the coefficients $\alpha_{j,\nu}$ and $\alpha_{\ell}$ in (\ref{f3spline}) are already determined by Corollary \ref{cor3}. 
We easily verify, that $\alpha_{j,\nu} \neq 0$. Further, $\sign({c}_{j}^{(2)}) = (-1)^{j+1}$, and therefore
$$ \alpha_{\ell} =  (-1)^{\ell+1} \textstyle\sum\limits_{\substack{j=1\\ j+\ell \, \text{even}}}^{n_{2}} a_{j}^{(3)} \left( \frac{x_{j,\ell}-x_{j,\ell-1}}{x_{\ell} - x_{j,\ell}} \right) \prod\limits_{k=1}^{{\ell-1}} \left(\frac{x_{k} - x_{j,k-1}}{x_{k} - x_{j,k}} \right), $$
where, by construction the term $\left( \frac{x_{j,\ell}-x_{j,\ell-1}}{x_{\ell} - x_{j,\ell}} \right) \prod\limits_{k=1}^{{\ell-1}} \left(\frac{x_{k} - x_{j,k-1}}{x_{k} - x_{j,k}} \right)$ has the sign $(-1)^{\ell}$. Therefore, we add in the sum either only negative or only positive terms, such that $\alpha_{\ell} \neq 0$ for $\ell=1, \ldots , n_{1}$.
Finally, for given $c^{(3)}, \, b^{(3)} \in {\mathbb R}$ we obtain  with (\ref{twof2}) and Lemma \ref{lemsigma}
\begin{align*} q_{1} &=  \textstyle \lim\limits_{t \to - \infty} \partial(f_{3}(t)) = c^{(3)} +  \sum\limits_{j=1}^{n_{2}} a_{j}^{(3)} \lim\limits_{t \to - \infty}\Big( \partial \sigma(f_{2,j}(t) \Big) = c^{(3)} -  \sum\limits_{j=1}^{n_{2}} a_{j}^{(3)} \sigma(-c_{j}^{(2)}) \\
%=c^{(3)} - \sum_{\substack{j=1\\ j \, \text{even}}}^{n_{2}} a_{j}^{(3)}\\
q_{0} &= \textstyle \lim\limits_{t \to - \infty} (f_{3}(t) - q_{1}t) = b^{(3)} + \sum\limits_{j=1}^{n_{2}} a_{j}^{(3)} \lim\limits_{t \to - \infty} (\sigma(f_{2,j}(t))+ \sigma(-c_{j}^{(2)})t) \\
&=  \textstyle b^{(3)} + \sum\limits_{j=1}^{n_{2}} a_{j}^{(3)} b_{j}^{(2)} \sigma(-c_{j}^{(2)}),
\end{align*}
where we have used that $\lim\limits_{t \to - \infty} \partial \sigma(f_{2,j}(t)) = -\sigma(-c_{j}^{(2)})$ and
$\lim\limits_{t \to - \infty} ( \sigma(f_{2,j}(t)) + \sigma(-c_{j}^{(2)})t ) = b_{j}^{(2)} \sigma(-c_{j}^{(2)})$.
\end{proof}

Theorem \ref{theomax1} implies that the model ${\mathcal Y}_{n_{1}, n_{2}}$  contains at least $(n_{1}+1
)(n_{2}+1) + n_{2} +1$ independent
real parameters  and $n_{2}$ sign parameters, namely $(n_{1}+1)(n_{2}+1)-1$ breakpoints, ${\mathbf A}^{(3)}$,  $b^{(3)}$, $c^{(3)}$, and ${\mathbf c}^{(2)}$. If we do not employ a source channel, then $f_{3}$ can still have the maximal number of $(n_{1}+1
)(n_{2}+1) -1$ breakpoints, as we have seen in Example \ref{ex1}, but these breakpoints cannot all be independently chosen.

\begin{example}\label{ex1a}
We consider again $f_{3} \in {\mathcal Y}_{3,3}$ in Example \ref{ex1}.
To define $f_{3}$ we had taken ${\mathbf c}^{(2)} = {\mathbf 0}$, i.e., no source channel. However, transforming the model (\ref{twol}) into (\ref{twolalt1}) 
provides 
$\tilde{\mathbf c}^{(2)} =(1, 0.5, -0.5)^{T}$ and 
$\tilde{\mathbf b}^{(2)} = (-0.5, -0.3, 0.2)^{T}$, 
where, similarly as in the proof of Lemma \ref{lemma1}
$$ \tilde{c}_{j}^{(2)} = c_{j}^{(2)} - \sum_{k=1}^{3} a_{j,k}^{(2)} \sigma(-a_{k}^{(1)}) = c_{j}^{(2)} - a_{j,2}^{(2)}-a_{j,3}^{(2)} = - a_{j,2}^{(2)}-a_{j,3}^{(2)}. $$
In other words, $\tilde{c}_{j}^{(2)}$ depends on the parameters in ${\mathbf A}^{(2)}$ and ${\mathbf A}^{(1)}$.
Therefore, the $15$ breakpoints are not longer completely independent. In this example  observe that 
$$ \textstyle \frac{(x_{1} - x_{j,0})}{(x_{1}-x_{j,1})} + \frac{(x_{1} - x_{j,0})(x_{2} - x_{j,1})}{(x_{1}-x_{j,1})(x_{2}-x_{j,2})} = 
1 + \frac{(x_{1} - x_{j,0})(x_{2} - x_{j,1})}{(x_{1}-x_{j,1})(x_{2}-x_{j,2})} + \frac{(x_{1} - x_{j,0})(x_{2} - x_{j,1})(x_{3} - x_{j,2})}{(x_{1}-x_{j,1})(x_{2}-x_{j,2})(x_{3}-x_{j,3})}
$$
and thus
$$ x_{j,3} = \textstyle x_{3} - \frac{(x_1-x_{j,0})(x_{2} - x_{j,1})(x_{3} - x_{j,2})}{(x_{2}-x_{j,2})(x_{j,1}-x_{j,0})} \qquad j=1,2,3.
$$
\end{example}

Generally, we find the following dependencies on breakpoints if the source channel in the model is not used, i.e., if ${\mathbf c}^{(2)} = {\mathbf 0}$.

\begin{lemma}\label{lem4}
For $n_{1}, n_{2} \ge 1$ let  $f_3 \in {\mathcal Y}_{n_{1}, n_{2}}$ in $(\ref{twol})$ with ${\mathbf c}^{(2)} = {\mathbf 0}$, and 
assume that $f_{3}$ possesses the maximal number of breakpoints. Then there exists an index set $ I \subset \{1, \ldots , n_{1}\}$ with $I \neq \emptyset$ and $I \neq \{1, \ldots , n_{1}\}$ such that we have for all $j=1, \ldots, n_{2}$,
\begin{equation}\label{xrel}
 \textstyle \sum\limits_{k \in I} \prod\limits_{\ell=1}^{k-1} \left( \frac{x_{\ell} - x_{j,\ell-1}}{x_{\ell} - x_{j,\ell}} \right)  = 1+ \sum\limits_{k \in I} \prod\limits_{\ell=1}^{k} \left( \frac{x_{\ell} - x_{j,\ell-1}}{x_{\ell} - x_{j,\ell}} \right),
\end{equation}
with the convention that $\prod\limits_{\ell=1}^{0} \left( \frac{x_{\ell} - x_{j,\ell-1}}{x_{\ell} - x_{j,\ell}} \right) =1$.
\end{lemma}

\begin{proof} We use here the notations as in Theorem \ref{theo1}.
As shown in the proof of Theorem \ref{theo1}, we can always rewrite $f_{3}$ into a model of the form (\ref{twolalt1}), where the components of $\tilde{\mathbf c}^{(2)}$ satisfy (\ref{bc}), i.e.,
\begin{equation}\label{index} \tilde{c}_{j}^{(2)} = - \sum_{\substack{k=1\\ a_{k}^{(1)}<0}}^{n_{1}} a_{j,k}^{(2)} |a_{k}^{(1)}| = 
- \sum_{k \in I} \tilde{a}_{j,k}^{(2)}, 
\end{equation}
and where $(\tilde{a}_{j,k}^{(2)})_{k=1}^{n_{1}}$ is obtained from  $({a}_{j,k}^{(2)} |a_{k}^{(1)}|)_{k=1}^{n_{1}}$ by a permutation (to reorder the breakpoints $x_{k}= -b^{(1)}_{k}/a_{k}^{(1)}$) and $I \subset \{1, \ldots , n_{1}\}$. 
By assumption, all $x_{j,\nu}= -\frac{\eta_{j,\nu}}{\mu_{j,\nu}}$ with $\mu_{j,\nu}$ and $\eta_{j,\nu}$ in (\ref{munu}) are well defined, i.e., $\mu_{j,\nu} \neq 0$. 
In particular,  $\tilde{c}_{j}^{(2)} = \mu_{j,0}\neq 0$ for all $j$. Therefore, the index set $I$ (corresponding to permuted negative  components of $a_{k}^{(1)}$) in (\ref{index})  is not empty.  
Further, since $\mu_{j,n_{1}} = \mu_{j,0} + \sum_{k=1}^{n_{1}} \tilde{a}_{j,k}^{(2)} \neq 0$, the index set  $I$ cannot be equal to $\{1, \ldots , n_{1}\}$.
Using that $\tilde{a}_{j,k}^{(2)} = \mu_{j,k} - \mu_{j,k-1}$ together with (\ref{rec1}) (with $c^{(2)}_{j}$ replaced by $\tilde{c}_{j}^{(2)}$), we obtain 
\begin{align*}
 \mu_{j,0} &=  \textstyle -\sum\limits_{k \in I} \tilde{a}_{j,k}^{(2)} = -\sum\limits_{k \in I} (\mu_{j,k} - \mu_{j,k-1})  \\
 &=  \textstyle -\sum\limits_{k \in I} \mu_{j,0}  \left( \prod\limits_{\ell=1}^{k} \left( \frac{x_{\ell} - x_{j,\ell-1}}{x_{\ell} - x_{j,\ell}} \right) -
  \prod\limits_{\ell=1}^{k-1} \left( \frac{x_{\ell} - x_{j,\ell-1}}{x_{\ell} - x_{j,\ell}} \right)\right).
\end{align*}
Thus,
$$  \textstyle \sum\limits_{k \in I} \prod\limits_{\ell=1}^{k-1} \left( \frac{x_{\ell} - x_{j,\ell-1}}{x_{\ell} - x_{j,\ell}} \right)  =  1+ \sum\limits_{k \in I}\prod\limits_{\ell=1}^{k} \left( \frac{x_{\ell} - x_{j,\ell-1}}{x_{\ell} - x_{j,\ell}} \right). $$
\end{proof}

%Our observations in Lemma \ref{lem3} and Example \ref{ex1} show that  one cannot  choose all $n^{2} + 2n$ breakpoints arbitrarily.

\begin{remark}
\label{remred}
1. The redundancy relation (\ref{xrel}) occurs in Example \ref{ex1} with $I=\{2,3\}$.

2. Without the source channel, i.e. for ${\mathbf c}^{(2)} = {\mathbf 0}$, the largest possible number of $(n_{1}+1)(n_{2}+1)-1$ breakpoints cannot by achieved for $n_{1}<3$ or $n_{2}<2$. For $n_{1}=1$ we cannot find a subset $I \subset \{1\}$, which satisfies $I \neq \emptyset$ and $I \neq \{ 1\}$. For $n_{1} =2$ we still cannot satisfy the redundancy relation (\ref{xrel}). For $I=\{1\}$ we immediately  get a contradiction from $ 1 = 1 + \left( \frac{x_{1} - x_{j,0}}{x_{1} - x_{j,1}} \right)$. For $I=\{2\}$, the relation (\ref{xrel}) leads to 
$$ \textstyle x_{1,0} = x_{1} - \frac{(x_{2}-x_{1,2})(x_{1}- x_{1,1})}{(x_{1,1} - x_{1,2})} $$
but $x_{1,0} > x_{1}$ contradicts the requirement $x_{1,0} \in (-\infty, x_{1})$. 
For $n_{2}=1$ we cannot realize the needed sign change for $\mu_{j,0}$, see the proof of Corollary \ref{cor3}.
\end{remark}

\subsection{Independent parameter set for the ReLU model with two hidden layers}
\label{secred}

The model ${\mathcal Y}_{n_{1}, n_{2}}$ in  (\ref{twol}) with two hidden layers  contains $n_{1}n_{2} + 2n_{1} + 3 n_{2} + 2$ parameters. 
As seen in Corollary \ref{corscaling} it can be always represented with $n_{1}n_{2} + n_{1} + 2n_{2} + 2 $ real parameters and $n_{2}$ sign parameters.
Using the observations in Subsection \ref{sec:break}, we can show that this set of parameters is independent.

\begin{theorem}\label{theored1}
Any function $f_{3} \in {\mathcal Y}_{n_{1},n_{2}}$ in $(\ref{twol})$ can be represented as a function of the form $(\ref{twolalt})$, i.e., 
\begin{align}\nonumber f_{3}(t) &=  {c}^{(3)} t + {b}^{(3)} + {\mathbf A}^{(3)} \sigma({\mathbf A}^{(2)} \sigma(t {\mathbf 1} + {\mathbf b}^{(1)}) + t {\mathbf c}^{(2)}  + {\mathbf b}^{(2)})\\
\label{twol3a}
&=  \textstyle {c}^{(3)} t + {b}^{(3)} + \sum\limits_{j=1}^{n_{2}} {a}^{(3)}_{j} \sigma\left( \sum\limits_{k=1}^{n_{1}} {a}^{(2)}_{j,k} \, \sigma(t + {b}^{(1)}_{k}) + {c}_{j}^{(2)} t + {b}^{(2)}_{j} \right), 
\end{align}
with a sign vector ${\mathbf c}^{(2)}$, i.e., $c_{j}^{(2)} \in \{-1,0, 1\}$ and with $n_{1}n_{2} + 2n_{2}+ n_{1} + 2$ real parameters ${c}^{(3)}$, ${b}^{(3)}$, ${a}_{j}^{(3)}$, ${a}_{j,k}^{(2)}$, ${b}_{k}^{(1)}$, ${b}_{j}^{(2)}$, $j=1, \ldots , n_{2}$, $k=1, \ldots , n_{1}$.
Moreover, these parameters  are independent, i.e., any restriction of ${\mathcal Y}_{n_{1},n_{2}}$ to a model $\tilde{\mathcal Y}_{n_{1},n_{2}}$, where one or more of these parameters are a priori fixed, leads to 
$$ \tilde{\mathcal Y}_{n_{1},n_{2}} \subsetneq {\mathcal Y}_{n_{1},n_{2}}. $$
\end{theorem}

\begin{proof}
The representation (\ref{twol3a}) of the functions $f_{3}$ in the model ${\mathcal Y}_{n_{1}, n_{2}}$ follows already from Corollary \ref{corscaling}.
On the other hand, as shown in Theorem \ref{theomax}, we can construct a function $f_{3}$ with $n_{1}n_{2} + n_{1} + n_{2}$ arbitrarily chosen breakpoints in ${\mathbb R}$, and in that model, we can still choose ${\mathbf A}^{(3)} \in {\mathbb R}^{1 \times n_{2}}$ and the sign vector ${\mathbf c}^{(2)}$. The restrictions on  $\sign({\mathbf A}^{(3)}) $ considered in the proof of Theorem \ref{theomax1} ensure that the function $f_{3}$ indeed possesses the maximal number of active breakpoints. If these conditions are not satisfied then we still obtain a function $f_{3} \in {\mathcal Y}_{n_{1}, n_{2}}$, where however not a maximal number of breakpoints may be active. Thus, the model possesses at least 
$n_{1}n_{2} + 2n_{2} + n_{1} + 2$ independent real parameters and one sign vector of length $n_{2}$. 
Indeed, if one of the  parameters is fixed a priori, then, as shown in Theorem \ref{theomax1}, we cannot longer  construct  all functions $f_{3} \in {\mathcal Y}_{n_{1},n_{2}}$  with $(n_{1}+1
)(n_{2}+1)-1$  arbitrarily prescribed  breakpoints.
\end{proof}

We finish this section by presenting  an example that shows  that  we can construct a function $f_{3} \in {\mathcal Y}_{n_{1},n_{2}}$
with $n_{1}n_{2} + n_{1}$ pre-determined  breakpoints without using the source channel, i.e., with ${\mathbf c}^{(2)} = {\mathbf 0}$ and $c^{(3)} =0$. Note that by Lemma \ref{lem4}, this is the  maximal  possible number of breakpoints we can prescribe in this case.

\begin{example}\label{ex2}
Consider the case $n_1=3$ and $n_2=2$ such that $n_1n_2+n_1=9$. We will construct $f_{3}$ in (\ref{twol2})  with breakpoints $x_k=k$ for $k=1,\ldots,9$. To this end we fix
\begin{align*}
x_1^{(1)} &= 1, & x_2^{(1)} &= 4, & x_3^{(1)} &= 7,\\
x_{1,1}^{(2)} &= 2, & x_{1,2}^{(2)} &= 5, & x_{1,3}^{(2)} &= 8,\\
x_{2,1}^{(2)} &= 3, & x_{2,2}^{(2)} &= 6, & x_{2,3}^{(2)} &= 9.
\end{align*}
Further, we choose
% $\mathbf{A}^{(1)}=(1,1,1)^T$, 
${\mathbf x} = \mathbf{x}^{(1)}=(x_k^{(1)})_{k=1}^3=(1,4,7)^T$, ${\mathbf c}^{(2)} = {\mathbf 0}$, $c^{(3)} = b^{(3)} =0$. Thus  we have $\mu_{j,0}=0$ for $j=1,2$.  
Let $(a_{1,1}^{(2)}, a_{2,1}^{(2)})^{T}  = (-1,1)^{T}$.
We can  now use (\ref{murec}) with initial values $\mu_{1,1} = a_{1,1}^{(2)} = -1$, $\mu_{2,1} = a_{2,1}^{(2)} = 1$, which  leads similarly as in (\ref{ajk2} to the recursion
$$  \textstyle a_{j,k}^{(2)} = \mu_{j,k} - \mu_{j,k-1} = a_{1,1}^{(2)} \left( \frac{x_{j,k}^{(2)}-x_{j,k-1}^{(2)}}{x_{j}^{(1)}- x_{j,k}^{(2)}} \right) \prod\limits_{\ell=2}^{k-1} \left( \frac{x_{\ell}^{(1)} - x_{j,\ell-1}^{(2)}}{x_{\ell}^{(1)} - x_{j,\ell}^{(2)}} \right), \quad j=1,2, \; k=2,3. $$
to determine the further components of ${\mathbf A}^{(2)}$. We obtain 
$$ {\mathbf A}^{(2)} = \begin{pmatrix}
-1 & 3 & -6\\
1 & -\frac{3}{2} & \frac{3}{4}
\end{pmatrix}.
$$
%\begin{align*}
%a_{1,1}^{(2)} &= -1,\\
%a_{1,2}^{(2)} &=  \textstyle a_{1,1}^{(2)} \left(\frac{x_{1,2}^{(2)}-x_{1,2-1}^{(2)}}{x_{2}^{(1)}-x_{1,2}^{(2)}} \right) \prod\limits_{\ell=2}^{2-1} \left( \frac{x_{\ell}^{(1)}- x_{1,\ell-1}^{(2)}}{x_{\ell}^{(1)} - x_{1,\ell}^{(2)}} \right) = (-1)\frac{5 - 2}{4 - 5} = 3,\\
%a_{1,3}^{(2)} &=  \textstyle a_{1,1}^{(2)} \left(\frac{x_{1,3}^{(2)}-x_{1,3-1}^{(2)}}{x_{3}^{(1)}-x_{1,3}^{(2)}} \right) \prod\limits_{\ell=2}^{3-1} \left( \frac{x_{\ell}^{(1)}- x_{1,\ell-1}^{(2)}}{x_{\ell}^{(1)} - x_{1,\ell}^{(2)}} \right) = (-1)\frac{8 - 5}{7 - 8}\frac{4 - 2}{4 - 5} = -6,\\
%a_{2,1}^{(2)} &= 1,\\
%a_{2,2}^{(2)} &= \textstyle a_{2,1}^{(2)} \left(\frac{x_{2,2}^{(2)}-x_{2,2-1}^{(2)}}{x_{2}^{(1)}-x_{2,2}^{(2)}} \right) \prod\limits_{\ell=2}^{2-1} \left( \frac{x_{\ell}^{(1)}- x_{2,\ell-1}^{(2)}}{x_{\ell}^{(1)} - x_{2,\ell}^{(2)}} \right) = 1\frac{6 - 3}{4 - 6} = \textstyle -\frac{3}{2},\\
%a_{2,3}^{(2)} &= \textstyle a_{2,1}^{(2)} \left(\frac{x_{2,3}^{(2)}-x_{2,3-1}^{(2)}}{x_{3}^{(1)}-x_{2,3}^{(2)}} \right) \prod\limits_{\ell=2}^{3-1} \left( \frac{x_{\ell}^{(1)}- x_{2,\ell-1}^{(2)}}{x_{\ell}^{(1)} - x_{2,\ell}^{(2)}} \right) = 1\frac{9 - 6}{7 - 9}\frac{4 - 3}{4 - 6} = \frac{3}{4},
%\end{align*}
Further, we take
\begin{align*}
b_1^{(2)} &= a_{1,1}^{(2)}(x_1^{(1)}-x_{1,1}^{(2)}) = -1(1-2)=1,\\
b_2^{(2)} &= a_{2,1}^{(2)}(x_1^{(1)}-x_{2,1}^{(2)}) = 1(1-3)=-2.
\end{align*}
to ensure that $x_{j,1}^{(2)} = - \frac{\eta_{j,1}}{\mu_{j,1}} = -\frac{(b_{1}^{(2)} -a_{j,1}^{(2)} x_{1}^{(1)})}{a_{j,1}^{(2)}}$, $j=1,2$.
Finally, set $\mathbf{A}^{(3)}=(1,1)$. Thus we obtain
\begin{align*}
f_3(t) &= \textstyle (1,1)\sigma\left(\begin{pmatrix}
-1 & 3 & -6\\
1 & -\frac{3}{2} & \frac{3}{4}
\end{pmatrix}\sigma\left(\begin{pmatrix}
1\\
1\\
1
\end{pmatrix}t-\begin{pmatrix}
1\\
4\\
7
\end{pmatrix}\right)+\begin{pmatrix}
1\\
-2
\end{pmatrix}\right)\\
&= \sigma\left(-\sigma(t-1)+3\sigma(t-4)-6\sigma(t-7)+1\right)\\
&\quad+ \sigma\Big(\sigma(t-1)-\frac{3}{2}\sigma(t-4)+\frac{3}{4}\sigma(t-7)-2\Big).
\end{align*}
Note that $2,5,8$ are the zeros of $-\sigma(t-1)+3\sigma(t-4)-6\sigma(t-7)+1$ and $3,6,9$ are the zeros of $\sigma(t-1)-\frac{3}{2}\sigma(t-4)+\frac{3}{4}\sigma(t-7)-2$, see Figure \ref{fig2}.

\begin{figure}
\begin{center}
	\includegraphics[scale=0.35]{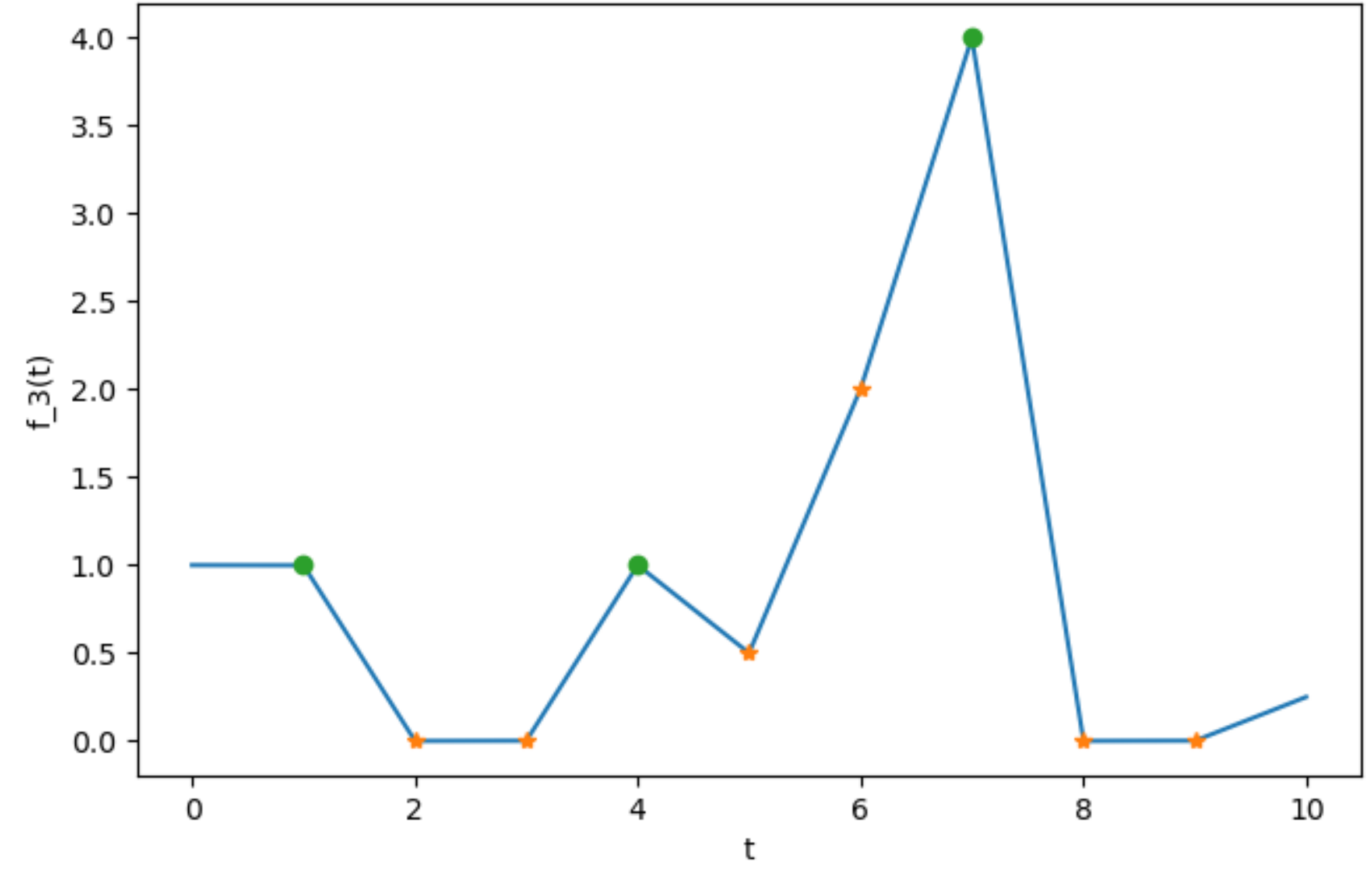} 
\end{center}
\caption{Illustration of $f_{3}$ in Example \ref{ex2} with $ $ prescribed knots.}
\label{fig2}
\end{figure}
\end{example}

\section{ReLU NN for three  and more hidden layers}
\label{sec:more}

In the general case $L\ge 3$, we consider the ReLU NN model ${\mathcal Y}_{n_{1},n_{2}, \ldots, n_{L-1}}$ with $L-1$ hidden layers  of widths $n_{0}= n_{L} =1$ and $n_{1}, \ldots , n_{L-1} \in {\mathbb N}$ in (\ref{DNN}), i.e.,
\begin{align} \label{LL}
f_{L}(t) &= t {\mathbf c}^{(L)} + {\mathbf b}^{(L)} + {\mathbf A}^{(L)} \sigma \Big( \ldots \sigma \Big( {\mathbf A}^{(2)}\sigma( t {\mathbf A}^{(1)}  + {\mathbf b}^{(1)}) + t {\mathbf c}^{(2)}+ {\mathbf b}^{(2)} \Big) + \ldots 
\Big) \\
\nonumber
&= \textstyle tc^{(L)}  + b^{(L)}+ \sum\limits_{k_{L-1}=1}^{n_{L-1}} \!\!a_{1,k_{L-1}}^{(L)} \sigma \Big(
\! \ldots  \sigma \Big(\sum\limits_{k_{1}=1}^{n_{1}} \!\!a_{k_{2},k_{1}}^{(2)} \sigma(a_{k_{1},1}^{(1)}t + b_{k_{1}}^{(1)}) + t c_{k_{2}}^{(2)} +  b_{k_{2}}^{(2)} \Big) + \! \ldots 
\Big) ,
\end{align}
where ${\mathbf A}^{(\ell)} = (a_{j,k}^{(\ell)})_{j,k=1}^{n_{\ell},n_{\ell-1}} \in\mathbb{R}^{n_\ell\times n_{\ell-1}}$ , $\ell=1,\ldots,L$, ${\mathbf b}^{(1)}\in{\mathbb R}^{n_{1}}$ and ${\mathbf b}^{(\ell)}, \, {\mathbf c}^{(\ell)} \in {\mathbb R}^{n_{\ell}}$, $\ell=2,\ldots , L$. 
In particular, ${\mathbf c}^{(L)}= c^{(L)} \in {\mathbb R}$ and 
${\mathbf b}^{(L)}= b^{(L)} \in {\mathbb R}$, since ${n}_{L}=1$.
Similarly as in Section \ref{sec:two}, we use the positive scaling property and set ${\mathbf A}^{(1)} = {\mathbf 1} \in {\mathbb R}^{n_{1}}$ and ${\mathbf c}^{(\ell)} = \sign({\mathbf c}^{(\ell)}) \in \{-1,0,1\}^{n_{\ell}}$, $\ell=2, \ldots, L-1$.
Then ${\mathcal Y}_{n_{1},n_{2}, \ldots, n_{L-1}}$ is determined by $n_{1}+ (n_{1}+1)n_{2}+ (n_{2}+ 1)n_{3}+ \ldots +(n_{L-2}+1)n_{L-1}+(n_{L-1}+1)n_{L} + 1$ real parameters and $n_{2}+n_{3}+ \ldots + n_{L-1}$ sign parameters.
The recursive structure of this model implies  that 
\begin{align}\label{Lrec}
f_{L}(t)  = t c^{(L)} + b^{(L)}
 + {\mathbf A}^{(L)} \sigma \Big( f_{L-1, j}(t) \Big)_{j=1}^{n_{L-1}} = t c^{(L)}  + b^{(L)} + \sum\limits_{j=1}^{n_{L-1}} a_{1,j}^{(L)} \sigma \Big( f_{L-1,j}(t) \Big)
\end{align}
where, with ${\mathbf A}_{j}^{(L-1)}$ denoting  the $j$-th row of ${\mathbf A}^{(L-1)}$,
\begin{align}\nonumber
& f_{L-1,j}(t) =  t c_{j}^{(L-1)} + b_{j}^{(L-1)} +{\mathbf A}^{(L-1)}_{j} \sigma \Big( \ldots \sigma \Big( {\mathbf A}^{(2)}\sigma(t {\mathbf A}^{(1)}  + {\mathbf b}^{(1)}) + t {\mathbf c}^{(2)}  + {\mathbf b}^{(2)} \Big) + \ldots  \Big)  \\
\label{flkl}
&= \textstyle \!t c_{j}^{(L-1)} \!+\! b_{j}^{(L-1)}+ \!\! \sum\limits_{k_{L-2}=1}^{n_{L-2}} \!\!a_{j,k_{L-2}}^{(L-1)} \sigma \Big(
%\sum\limits_{k_{L-2}=1}^{n_{L-2}} a_{k_{L-1},k_{L-2}}^{(L-1)} \sigma\Big( 
\ldots \sigma \Big(\sum\limits_{k_{1}=1}^{n_{1}} \!\!a_{k_{2},k_{1}}^{(2)} \sigma(a_{k_{1},1}^{(1)}t + b_{k_{1}}^{(1)}) + t  c_{k_{2}}^{(2)}  + b_{k_{2}}^{(2)} \Big) + \ldots \Big)  
%\nonumber
\end{align}
are $n_{L-1}$ functions in ${\mathcal Y}_{n_{1},n_{2}, \ldots , n_{L-2}}$. 

\subsection{Representation of the DNN model as a continuous linear spline function model} 

We will employ an induction argument  to represent  $f_{L} \in {\mathcal Y}_{n_{1},n_{2}, \ldots , n_{L-1}}$  as  a CPL  spline function in $\Sigma_{N_{L}}$ with at most $N_{L,\max}:=\prod\limits_{\ell=1}^{L-1} (n_{\ell} +1) -1$ breakpoints.

\begin{theorem}\label{theo4.1}
Let $L \ge 2$. Then a function $f_{L} \in {\mathcal Y}_{n_{1}, \ldots , n_{L-1}}$ in $(\ref{DNN})$ resp.\ $(\ref{LL})$ can be represented as a \textnormal{CPL} spline function
\begin{equation}\label{4.4}
f_{L}(t) = (q_{1}^{(L)} t + q_{0}^{(L)}) + \sum_{k=1}^{N_{L}} \alpha_{k}^{(L)} \sigma(t - \xi_{k}^{(L)}), 
\end{equation}  
i.e., $f_{L} \in \Sigma_{N_{L}}$ with $N_{L} \le N_{L,\max}$, 
$\alpha_{k}^{(L)} \in {\mathbb R}$, $k=1, \ldots , N_{L}$, ordered breakpoints $\xi_{1}^{(L)} < \xi_{2}^{(L)} < \ldots < \xi_{N_{L}}^{(L)}$, and
$ q_{1}^{(L)} := \lim\limits_{t\to - \infty} f_{L}'(t)$, $q_{0}^{(L)} := \lim\limits_{t\to - \infty} (f_{L}(t)- q_{1}^{(L)} t)$. 
%Note that the breakpoint $\xi_{k}^{(L)}$ is only active if $\alpha_{k}^{(L)} \neq 0$.
\end{theorem}

\begin{proof}
 For $L=2$ and $L=3$,  the representation (\ref{4.4}) follows from  Lemma \ref{lemma1} and Theorem \ref{theo1}, respectively.  Assume now, that  we have shown  (\ref{4.4}) for  functions in ${\mathcal Y}_{n_{1},n_{2}, \ldots, n_{L-2}}$. 
Then all functions $f_{L-1,j}$ in (\ref{flkl}) can be represented as
$$ f_{L-1,j}(t) =  (q^{(L-1)}_{1,j} t + q^{(L-1)}_{0,j}) + \sum_{k=1}^{N_{L-1}} \alpha_{j,k}^{(L-1)}\,  \sigma(t- \xi_{k}^{(L-1)}), $$
with $N_{L-1} \le N_{L-1,\max}=\prod\limits_{\ell=1}^{L-2}(n_{\ell}+1)-1$, and in particular, all functions $f_{L-1,j}$, $j =1, \ldots , n_{L-1}$, possess  the same  breakpoints $\xi_{k}^{(L-1)}$, $k=1, \ldots , N_{L-1}$, (which may not all be active).  Then (\ref{Lrec}) implies
$$ f_{L}(t) =  t c^{(L)}  + b^{(L)} + \sum_{j=1}^{n_{L-1}} a_{1,j}^{(L)} \sigma \Big((q^{(L-1)}_{1,j} t + q^{(L-1)}_{0,j}) + \sum_{k=1}^{N_{L-1}} \alpha_{j,k}^{(L-1)} \, \sigma(t- \xi_{k}^{(L-1)}) \Big) $$
and application of Theorem  \ref{theo1} (with $n_{1}= N_{L-1}$ and $n_{2}= n_{L-1}$)  yields the assertion, where 
$$ N_{L} \le (N_{L-1,\max} +1)(n_{L-1}+1)-1 = \Big(\prod\limits_{\ell=1}^{L-2}(n_{\ell} + 1) -1+1 \Big)(n_{L-1}+1) - 1 = N_{L,\max}. $$
The values for $q_{1}^{(L)}$ and $q_{0}^{(L)}$ follow directly from the observation that $\sum\limits_{k=1}^{N_{L}} \alpha_{k}^{(L)} \sigma(t - \xi_{k}^{(L)}) =0$ for $t < \xi_{1}^{(L)}$.
\end{proof}

More exactly, we can derive a recursion for the coefficients of the representation of $f_{L}$ in (\ref{4.4}).

\begin{corollary}\label{corLL}
Let $L \ge 2$. Let a function $f_{L} \in {\mathcal Y}_{n_{1}, \ldots , n_{L-1}}$ be given in the form 
$$ f_{L}(t) = t c^{(L)}  + b^{(L)} + \sum\limits_{j=1}^{n_{L-1}} a_{1,j}^{(L)} \sigma \Big( f_{L-1,j}(t) \Big) $$
with 
$$f_{L-1,j}(t)= (q^{(L-1)}_{1,j} t + q^{(L-1)}_{0,j}) + \sum_{\ell=1}^{N_{L-1}} \alpha_{j,\ell}^{(L-1)} \sigma(t - \xi_{\ell}^{(L-1)}), \qquad j=1, \ldots, n_{L-1} $$ 
and ordered breakpoints $\xi_{1}^{(L-1)} < \xi_{2}^{(L-1)} < \ldots < \xi_{N_{L-1}}^{(L-1)}$. 
 Then  $f_{L}(t)$ can be represented in the form
$$ f_{L}(t) = q_{1}^{(L)} t + q_{0}^{(L)} + \sum_{\ell=1}^{N_{L-1}} \beta_{\ell}^{(L)} \sigma(t - \xi_{\ell}^{(L-1)}) + \sum_{j=1}^{n_{L-1}} \sum_{\nu=0}^{N_{L-1}} \beta_{j,\nu}^{(L)} \, \sigma(t- \xi_{j,\nu}^{(L)}),
$$
where $\xi_{\ell}^{(L-1)}$ are the breakpoints given already by $f_{L-1,j}(t)$, 
and with
\begin{align*}
\xi_{j,\nu}^{(L)} &:= \left\{ \begin{array}{ll}
- \frac{\eta_{j,\nu}^{(L-1)}}{\mu_{j,\nu}^{(L-1)}} & \mu_{j,\nu}^{(L-1)} \neq 0 \\
-\infty & \text{otherwise} \end{array} \right. , \quad j=1, \ldots , n_{L-1}, \nu=0, \ldots , N_{L-1}, \\
\beta_{j,\nu}^{(L)} &:= a_{1,j}^{(L)} | \mu_{j,\nu}^{(L-1)}|  \, \chi_{(\xi_{\nu}^{(L-1)}, \xi_{\nu+1}^{(L-1)})} (\xi_{j,\nu}^{(L)}),  \quad j=1, \ldots , n_{L-1}, \nu=0, \ldots , N_{L-1}, \\
\beta_{\ell}^{(L)} &:= \sum_{j=1}^{n_{L-1}} a_{1,j}^{(L)} \Big( \alpha_{j,\ell}^{(L-1)} \, \chi_{(0,\infty)}(f_{L-1,j} (\xi_{\ell}^{(L-1)})) \\
 &  \qquad \qquad + 
[\sigma(\mu_{j,\ell}^{(L-1)}) + \sigma(-\mu_{j,\ell-1}^{(L-1)})] \chi_{\{0\}}(f_{L-1,j} (\xi_{\ell}^{(L-1)})) \Big)
, \qquad \ell=1, \ldots , N_{L-1}, \\
q_{1}^{(L)} &:= c^{(L)} - \sum_{j=1}^{n_{L-1}} a_{1,j}^{(L)} \sigma(-q_{1,j}^{(L-1)}), \\
q_{0}^{(L)} &:= b^{(L)} + \sum_{j=1}^{n_{L-1}} a_{1,j}^{(L)} \left(q_{0,j}^{(L-1)}\, \chi_{(-\infty, 0)}(q_{1,j}^{(L-1)}) + \sigma(q_{0,j}^{(L-1)}) \chi_{\{0 \}} (q_{1,j}^{(L-1)}) \right),
\end{align*}
where
$$ \mu_{j,\nu}^{(L-1)} := q_{1,j}^{(L-1)} + \sum_{r=1}^{\nu} \alpha_{j,r}^{(L-1)}, \quad \eta_{j,\nu}^{(L-1)} := q_{0,j}^{(L-1)} - \sum_{r=1}^{\nu} \alpha_{j,r}^{(L-1)} \xi_{r}^{(L-1)}. $$
%and where $\chi_{T}$ denotes the characteristic  function of $T \subset {\mathbb R}$.
\end{corollary}

\begin{remark} \label{rem41}
Instead of normalizing  ${\mathbf c}^{(\ell)} = \sign ({\mathbf c}^{(\ell)})$ we can also normalize iteratively in the recursive representation of $f_{L}$ such that $q_{1,j}^{(\ell)} = \sign(q_{1,j}^{(\ell)})$ for $j=1, \ldots , n_{\ell}$ and $\ell= 1, \ldots, L-1$. This normalization will be applied in Theorem \ref{theo4.2}.
\end{remark}

The recursive application of Corollary \ref{corLL} also provides us an algorithm  to compute the possible breakpoints of $f_{L}$, see Algorithm \ref{algo2}.

\begin{algorithm}[h!]\caption{Transfer from (\ref{LL}) to (\ref{4.4})}
\label{algo2}
\small{
\textbf{Input:} ${\mathbf n}=(n_{1}, \ldots , n_{L-1})^{T} \in {\mathbb N}^{L-1}$, $n_{0}=n_{L}=1$, ${\mathbf A}^{(\ell)}=(a_{j,k}^{(\ell)})_{j,k=1}^{n_{\ell}, n_{\ell-1}} \in {\mathbb R}^{n_{\ell} \times n_{\ell-1}}$, ${\mathbf b}^{(1)}\in\mathbb{R}^{n_1}$,\\
\null \qquad \quad \; ${\mathbf b}^{(\ell)},{\mathbf c}^{(\ell)} \in {\mathbb R}^{n_\ell}$, $\ell=2, \ldots , L$.

\begin{description}
\item{1.} for $j=1:n_{2}$ do \\
Use Algorithm \ref{alg1} to obtain from $c_j^{(2)}$, $b_j^{(2)}$, $(a_{j,k}^{(2)})_{k=1}^{n_1}$, $(a_k^{(1)})_{k=1}^{n_1}$, $(b_k^{(1)})_{k=1}^{n_1}$ the parameters
$q_{1,j}^{(2)}$, $q_{0,j}^{(2)}$, $N_{2}$, $(\alpha_{j,\nu}^{(2)})_{\nu=1}^{N_{2}}$, $(\xi_{\nu}^{(2)})_{\nu=1}^{N_{2}}$ such that 
$$ \textstyle f_{2,j}(t)= c_{j}^{(2)}t + b_{j}^{(2)} + \sum\limits_{k=1}^{n_{1}} a_{j,k}^{(2)} \sigma(a_{k}^{(1)}t + b_{k}^{(1)}) =
(q_{1,j}^{(2)} t + q_{0,j}^{(2)}) + \sum\limits_{\nu=1}^{N_{2}} \alpha_{j,\nu}^{(2)} \sigma(t-\xi_\nu^{(2)}). $$
%from $c_j^{(2)}$, $b_j^{(2)}$, $(a_{j,k}^{(2)})_{k=1}^{n_1}$, $(a_k^{(1)})_{k=1}^{n_1}$, $(b_k^{(1)})_{k=1}^{n_1}$. 
%Here $N_{2} \le n_{1}$ is the number of pairwise distinct finite breakpoints. \\
\hspace*{-4mm} end(for($j$))
\item{2.} for $\ell=3:L$ do\\
%For $j=1:n_{\ell}$ do\\
Compute from given $(q_{1,j}^{(\ell-1)})_{j=1}^{n_{\ell-1}}$, $(q_{0,j}^{(\ell-1)})_{j=1}^{n_{\ell-1}}$, $(\alpha_{j,\nu}^{(\ell-1)})_{j,\nu=1}^{n_{\ell-1},N_{\ell-1}}$,
$(\xi_{\nu}^{(\ell-1)})_{\nu=1}^{N_{\ell-1}}$, 
${\mathbf A}^{(\ell)}$, ${\mathbf b}^{(\ell)}$, ${\mathbf c}^{(\ell)}$ the values 
%
%We have $$f_{\ell,j}(t)= (c^{(\ell)}_j t + b^{(\ell)}_j) + \sum_{\nu=1}^{n_{\ell}} a_{j,\nu}^{(\ell)} \sigma(f_{\ell-1,\nu}(t)), \qquad j=1, \ldots, n_{\ell} $$ 
%and want to write it in the form
%$$f_{\ell,j}(t)= (q^{(\ell)}_{1,j} t + q^{(\ell)}_{0,j}) + \sum_{\nu=1}^{N_{\ell}} \alpha_{j,\nu}^{(\ell)} \sigma(t - \xi_{\nu}^{(\ell)}), \qquad j=1, \ldots, n_{\ell} $$ 
\begin{align*}
\mu_{j,\nu}^{(\ell-1)} &:= \textstyle q_{1,j}^{(\ell-1)} + \sum\limits_{k=1}^{\nu} \alpha_{j,k}^{(\ell-1)}, \quad j=1, \ldots , n_{\ell-1}, \; \nu=1, \ldots , N_{\ell-1},\\
\eta_{j,\nu}^{(\ell-1)}&:= \textstyle q_{0,j}^{(\ell-1)} - \sum\limits_{k=1}^{\nu} \alpha_{j,k}^{(\ell-1)} \xi_{k}^{(\ell-1)}, \quad j=1, \ldots , n_{\ell-1}, \; \nu=1, \ldots , N_{\ell-1},\\
\xi_{j,\nu}^{(\ell)} &:=  \textstyle \left\{ \begin{array}{ll}
- \frac{\eta_{j,\nu}^{(\ell-1)}}{\mu_{j,\nu}^{(\ell-1)}} & \mu_{j,\nu}^{(\ell-1)} \neq 0 \\
-\infty & \text{otherwise} \end{array} \right. , \quad j=1, \ldots , n_{\ell-1}, \nu=0, \ldots , N_{\ell-1}, \\
\beta_{j,\nu,k}^{(\ell)} &:= \textstyle a_{k,j}^{(\ell)} | \mu_{j,\nu}^{(\ell-1)}|  \, \chi_{(\xi_{\nu}^{(\ell-1)}, \xi_{\nu+1}^{(\ell-1)})} (\xi_{j,\nu}^{(\ell)}),  \quad j=1, \ldots , n_{\ell-1}, \nu=0, \ldots , N_{\ell-1}, k=1,\ldots,n_\ell,\\
\beta_{\nu,k}^{(\ell)} &:= \textstyle \sum\limits_{j=1}^{n_{\ell-1}} a_{k,j}^{(\ell)} \Big( \alpha_{j,\nu}^{(\ell-1)} \, \chi_{(0,\infty)}(f_{\nu-1,j} (\xi_{\nu}^{(\ell-1)})) \\
&  \quad + 
[\sigma(\mu_{j,\nu}^{(\ell-1)}) + \sigma(-\mu_{j,\nu-1}^{(\ell-1)})] \chi_{\{0\}}(f_{\nu-1,j} (\xi_{\nu}^{(\ell-1)})) \Big)
, \quad \nu=1, \ldots , N_{\ell-1},  k=1,\ldots,n_\ell,\\
q_{1,k}^{(\ell)} &:= \textstyle c_k^{(\ell)} - \sum\limits_{j=1}^{n_{\ell-1}} a_{k,j}^{(\ell)} \sigma(-q_{1,j}^{(\ell-1)}), \quad k=1,\ldots,n_\ell,\\
q_{0,k}^{(\ell)} &:= \textstyle b_k^{(\ell)} + \sum\limits_{j=1}^{n_{\ell-1}} a_{k,j}^{(\ell)} \left(q_{0,j}^{(\ell-1)}\, \chi_{(-\infty, 0)}(q_{1,j}^{(\ell-1)}) + \sigma(q_{0,j}^{(\ell-1)}) \chi_{\{0 \}} (q_{1,j}^{(\ell-1)}) \right), \quad k=1,\ldots, n_\ell.
\end{align*}
Order the set of breakpoints 
$$(\{\xi_\nu^{(\ell-1)}: \nu=1, \ldots , N_{\ell-1}\} \cup \{\xi_{j,\nu}^{(\ell)}: \, j=1, \ldots, n_{\ell-1}, \, \nu=0, \ldots , N_{\ell-1} \}) \setminus \{ - \infty \}$$
 by size to obtain $\xi_1^{(\ell)}<\ldots<\xi_{N_\ell}^{(\ell)}$.\\
for $k=1:n_{\ell}$ do \\
\hspace*{3mm} Apply the same permutation to the coefficients $\beta_{j,\nu,k}^{(\ell)}$ (corresp.\ to $\xi_{j,\nu}^{(\ell-1)}$)
and $\beta_{\nu,k}^{(\ell)}$ (corresp.\\
\hspace*{3mm} to $\xi_{\nu}^{(\ell-1)}$)  to obtain $\alpha_{k,1}^{(\ell)},\ldots,\alpha_{k,N_\ell}^{(\ell)}$ (where coefficients corresp.\ to $-\infty$ are removed).\\
end(for($k$))\\
for $r=1:N_{\ell}$ do \\
\hspace*{3mm} if $(\alpha_{k,r}^{(\ell)})_{k=1}^{n_{\ell}} == {\mathbf 0}$ then remove $\xi_{r}^{(\ell)}$ from $(\xi_{\nu}^{(\ell)})_{\nu=1}^{N_{\ell}}$, remove the column $(\alpha_{k,r}^{(\ell)})_{k=1}^{n_{\ell}}$ from\\
\hspace*{3mm} $(\alpha_{k,\nu}^{(\ell)})_{k,\nu=1}^{n_{\ell},N_{\ell}}$, set $N_{\ell} := N_{\ell}-1$ end(if)\\
 end(for($r$))\\
\hspace*{-4mm} end(for($\ell$))
\end{description}
\textbf{Output:} $q_{1,1}^{(L)}$, $q_{0,1}^{(L)}$, $\xi_1^{(L)},\ldots,\xi_{N_L}^{(L)}$, $\alpha_{1,1}^{(L)},\ldots,\alpha_{1,N_L}^{(L)}$ to represent $f_{L}$ in (\ref{4.4}).
}
\end{algorithm}

\subsection{Construction of three hidden layer NN models with prescribed knots}
\label{sec3knots}

Similarly as before in the case of two hidden layers, 
 we will study the question, how to construct a function $f_{4} \in {\mathcal Y}_{n_{1},n_{2}, n_{3}}$  with 
$n_{2}n_{3} + n_{1}n_{2}+ n_{1}+ n_{2} + n_{3}$ predetermined breakpoints in ${\mathbb R}$.
This is the maximal number of predetermined breakpoints we can hope for, since, as shown in the beginning of Section \ref{sec:more}, the model ${\mathcal Y}_{n_{1},n_{2}, n_{3}}$ depends on $n_{2}n_{3} + n_{1}n_{2}+ n_{1}+ n_{2} + 2n_{3}+ 2$ real parameters, where the $n_{3}+2$ parameters represented by ${\mathbf A}^{(4)} \in {\mathbb R}^{1 \times n_{3}}$ and $c^{(4)}, \, b^{(4)} \in {\mathbb R}$ have no influence on the breakpoints of $f_{4}$.
Such a construction would therefore imply that  all real parameters involved  in ${\mathcal Y}_{n_{1},n_{2}, n_{3}}$ (after normalisation ${\mathbf A}^{(1)} = {\mathbf 1}$ and ${\mathbf c}^{(\ell)} = \sign({\mathbf c}^{(\ell)})$ or recursive normalization of ${\mathbf q}_{1}^{(\ell)}$ as in Remark \ref{rem41}) are independent.

\begin{theorem} \label{theo4.2}
Let $n_{1} \ge 1$, $n_{2} \ge 2$, $n_{3}\ge \log_{2}(n_{1}+(n_{1}+1)n_{2})$.
For $N=n_{2}n_{3}+ n_{1}n_{2}+n_{1} + n_{2}+ n_{3}$ and an arbitrary ordered  set of breakpoints,
$$ x_{1} < x_{2}  < \ldots 
 < x_{N},$$
there exists a function $f_{4} \in {\mathcal Y}_{n_{1},n_{2}, n_{3}}$ of the form $(\ref{LL})$  that possesses all these breakpoints $x_{\ell}$, $\ell=1, \ldots , N$. 
\end{theorem}

\begin{proof} We show the assertion by constructing a function $f_{4}(t)$ of the form 
$$
 f_{4}(t)
= \textstyle t c^{(4)} + b^{(4)}+ \!\sum\limits_{r=1}^{n_3 }a_r^{(4)} \sigma \!\!\left( \sum\limits_{j=1}^{n_{2}} a^{(3)}_{r,j} \sigma\!\left( \sum\limits_{k=1}^{n_{1}} a^{(2)}_{j,k} \, \sigma(a^{(1)}_{k} t + b^{(1)}_{k}) +t c^{(2)}_{j} \!\!+ b^{(2)}_{j} \!\right) \!\!+ t c_r^{(3)} \!\! + b_r^{(3)}\!\right) $$
with these $N$ prescribed breakpoints.

Let the given  breakpoints be  reordered  and denoted as
\begin{align} \nonumber
& x_{1}^{(1)} < x_{2}^{(1)} < \ldots < x_{n_{1}}^{(1)} \\
\label{points}
& x_{\nu}^{(1)} < x_{1,\nu}^{(2)} < x_{2,\nu}^{(2)} < \ldots x_{n_{2},\nu}^{(2)} < x_{\nu+1}^{(1)}, \qquad \nu=0, \ldots , n_{1}, \\
\nonumber
& x_{j,0}^{(2)}<x_{1,j}^{(3)} < x_{2,j}^{(3)} < \ldots < x_{n_3,j}^{(3)}<x_{j+1,0}^{(2)},\quad j=0,\dots,n_2,
\end{align}
where $x_{0}^{(1)} = x_{0,0}^{(2)}:=-\infty$, $x_{n_{1}+1}^{(1)} := \infty$ and $x_{n_{2}+1,0}^{(2)} := x_{1}^{(1)}$, see Figure \ref{figknots}.

\begin{figure}[h]
\begin{center}
	\includegraphics[scale=0.3]{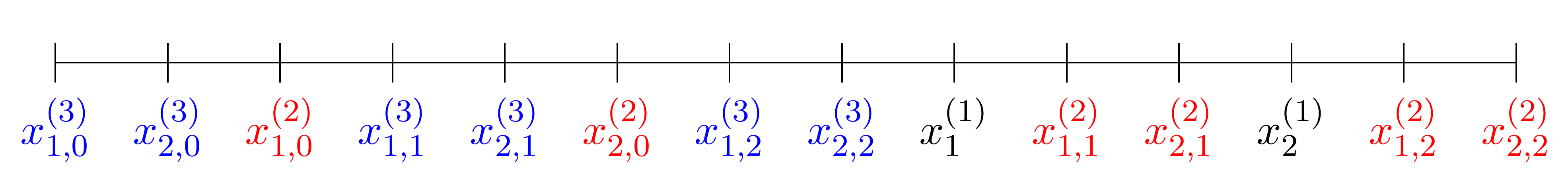}
\end{center}
\caption{\small Ordering of breakpoints for the first, second and third layer for $n_{1}=n_{2}=n_{3}=2$.}
\label{figknots}
\end{figure}

\noindent
1. We observe that $f_{4}(t) = t c^{(4)} + b^{(4)} + \sum\limits_{r=1}^{n_{3}} a_r^{(4)} \sigma(f_{3,r}(t))$, %with ${\mathbf A}^{(4)} = ((a_{r}^{(4)})_{r=1}^{n_{3}})^{T} \in {\mathbb R}^{1 \times n_{3}}$ and 
where each function $f_{3,r}(t)$, $r=1, \ldots , n_{3}$ is given by 
\begin{equation}\label{f3r1} f_{3,r}(t) =  \textstyle t c_{r}^{(3)} + b_{r}^{(3)} + \sum\limits_{j=1}^{n_{2}} a_{r,j}^{(3)} \sigma\left( \sum\limits_{k=1}^{n_{1}} a^{(2)}_{j,k} \, \sigma \left( t - x_{k}^{(1)} \right) + t c_{j}^{(2)} + b^{(2)}_{j} \right).
\end{equation}
We apply Theorem \ref{theomax1}  and choose
with $c_{j}^{(2)} := (-1)^{j+1}$, $j=1, \ldots, n_{2}$, $\mu_{j,0}^{(2)} := c_{j}^{(2)}$, 
\begin{align*} 
\mu_{j,k}^{(2)} &:= \textstyle c_{j}^{(2)} \prod\limits_{\ell=1}^{k} \Big( \frac{x_{\ell}^{(1)} - x_{j,\ell-1}^{(2)}}{x_{\ell}^{(1)} - x_{j,\ell}^{(2)}} \Big), \quad j=1, \ldots , n_{2}, k=1, \ldots, n_{1}, \\
a_{j,k}^{(2)} &:= \mu_{j,k}^{(2)} - \mu_{j,k-1}^{(2)} \quad j=1, \ldots , n_{2}, k=1, \ldots, n_{1}, \\
b_{j}^{(2)} &:= -x_{j,0}^{(2)} c_{j}^{(2)} = x_{j,0}^{(2)} (-1)^{j}, \qquad j=1, \ldots , n_{2}, \\
(a_{r,j}^{(3)})_{j=1}^{n_{2}} &\in {\mathbb R}^{n_{2}} \quad \text{with} \quad \sign (a_{r,j}^{(3)}) = \epsilon_{r} \, (-1)^{j}, \quad c_{r}^{(3)}, \, b_{r}^{(3)} \in {\mathbb R},
\end{align*}
where $\epsilon_{r} \in \{ -1,1\}$ and $a_{r,j}^{(3)}$, $c_{r}^{(3)}$, $b_{r}^{(3)}$ will be fixed later. Then the function $f_{3,r}$ in (\ref{f3r1}) possesses  the $n_{1}n_{2}+ n_{1}+n_{2}$ breakpoints $x_{\ell}^{(1)}$, $\ell=1, \ldots , n_{1}$, and  $x_{j,\nu}^{(2)}$, $\nu=0, \ldots , n_{1}$, $j=1, \ldots, n_{2}$.

More exactly, $f_{3,r}$ in (\ref{f3r1}) can be represented as
\begin{equation}\label{f3r}  f_{3,r}(t) = t q_{1,r}^{(3)} + q_{0,r}^{(3)}+ \sum_{\ell=1}^{n_{1}} \alpha_{r,\ell} {\sigma}(t-x_{\ell}^{(1)}) + 
\sum_{\nu=0}^{n_{1}} \sum_{j=1}^{n_{2}} \alpha_{r,j,\nu}{\sigma}(t- x_{j,\nu}^{(2)}), 
\end{equation}
where as in the proof of Theorem \ref{theomax1},
\begin{align*} 
\alpha_{r,j,\nu}  &=  \textstyle a_{r,j}^{(3)} |\mu_{j,\nu}^{(2)}| 
= a_{r,j}^{(3)} \left| \prod\limits_{\ell=1}^{\nu} \left( \frac{x_{\ell}^{(1)} - x_{j,\ell-1}^{(2)}}{x_{\ell}^{(1)} - x_{j,\ell}^{(2)}} \right) \right| \neq 0 \qquad \nu=0, \ldots , n_{1}, \, j=1, \ldots , n_{2}, \\
\alpha_{r,\ell} &= (-1)^{\ell+1}  \textstyle\sum\limits_{\substack{j=1\\ j+\ell \, \text{even}}}^{n_{2}} a_{r,j}^{(3)}  \Big( \frac{x_{j,\ell}^{(2)} - x_{j,\ell-1}^{(2)}}{x_{\ell}^{(1)} - x_{j,\ell}^{(2)}} \Big) \prod\limits_{k=1}^{\ell-1} \Big( \frac{x_{k}^{(1)} - x_{j,k-1}^{(2)}}{x_{k}^{(1)} - x_{j,k}^{(2)}} \Big)\neq 0 \quad \ell=1, \ldots , n_{1}, 
\end{align*}
i.e., all $(n_{1}+1)(n_{2}+1)-1$ breakpoints are active in $f_{3,r}$, $r=1, \ldots , n_{3}$, with $\sign(\alpha_{r,j,\nu}) = (-1)^{j} \epsilon_{r}$ and $\sign(\alpha_{r,\ell}) = (-1)^{\ell+1} \epsilon_{r}$.  

2. We observe now that in $(-\infty,x_{1}^{(1)})$ the functions $f_{3,r}$ in (\ref{f3r}) are of the form 
$$ f_{3,r}(t) = t q_{1,r}^{(3)} + q_{0,r}^{(3)}+ \sum_{j=1}^{n_{2}} \alpha_{r,j,0}{\sigma}(t- x_{j,0}^{(2)}) = t q_{1,r}^{(3)} + q_{0,r}^{(3)}+ \sum_{j=1}^{n_{2}} a_{r,j}^{(3)} \, {\sigma}(t- x_{j,0}^{(2)}), $$
where  we have used that $\alpha_{r,j,0}=a_{r,j}^{(3)} |\mu_{j,0}^{(2)}| =a_{r,j}^{(3)}$, and where 
\begin{align} \label{q1r}
q_{1,r}^{(3)} &= c_{r}^{(3)} - \sum_{j=1}^{n_{2}} a_{r,j}^{(3)} \sigma(-c_{j}^{(2)}) = c_{r}^{(3)} - \sum\limits_{\substack{j=1\\ j \, \text{even}}}^{n_{2}} a_{r,j}^{(3)}, \\
\label{q0r}
q_{0,r}^{(3)} &= b_{r}^{(3)} + \sum_{j=1}^{n_{2}} a_{r,j}^{(3)}  b_{j}^{(2)} \, \chi_{(-\infty,0)}(c_{j}^{(2)}) =  b_{r}^{(3)} +  \sum\limits_{\substack{j=1\\ j \, \text{even}}}^{n_{2}} a_{r,j}^{(3)} \, x_{j,0}^{(2)}. 
\end{align}
Thus, in $(-\infty,x_{1}^{(1)})$, we have $f_{3,r} \in {\mathcal Y}_{n_{2}}$  with breakpoints $x_{j,0}^{(2)}$, $j=1, \ldots , n_{2}$, and we can apply our observations from Section \ref{sec:break} to construct $f_{4}(t)$ possessing also the third-layer breakpoints  $x_{r,j}^{(3)}$, $r=1, \ldots , n_{3}$, $j=0, \ldots , n_{2}$ in $(-\infty,x_{1}^{(1)})$. To achieve this goal, we
 need to choose the parameters $c_{r}^{(3)}$, $b_{r}^{(3)}$ and $a_{r,j}^{(3)}$ to ensure that $f_{3,r}(x_{r,j}^{(3)})=0$ for $j=0, \ldots , n_{2}$. As in Section 2, we have 
$$ f_{3,r}(t) = \mu_{r,j}^{(3)} t + \eta_{r,j}^{(3)} \quad \text{for} \quad t \in [x_{j,0}^{(2)}, x_{j+1,0}^{(2)}]$$
with 
\begin{equation}\label{mjk} \textstyle \mu_{r,j}^{(3)} = q_{1,r}^{(3)} + \sum\limits_{\ell=1}^{j} a_{r,\ell}^{(3)}, \qquad  \eta_{r,j}^{(3)} = q_{0,r}^{(3)} - \sum\limits_{\ell=1}^{j} a_{r,\ell}^{(3)} x_{\ell,0}^{(2)}. 
\end{equation}
We choose now $q_{1,r}^{(3)}= \epsilon_{r}$, $q_{0,r}^{(3)}= -\epsilon_{r} x_{r,0}^{(3)}$, $\mu_{r,0}^{(3)} = q_{1,r}^{(3)} = \epsilon_{r}$, 
$$ \mu_{r,j}^{(3)} = \textstyle \mu_{r, j-1}^{(3)} \Big( \frac{x_{j,0}^{(2)} - x_{r,j-1}^{(3)}}{x_{j,0}^{(2)}- x_{r,j}^{(3)}} \Big), \quad r=1, \ldots , n_{3}, j=1, \ldots, n_{2}, $$
and
$$ a_{r,j}^{(3)} := \mu_{r,j}^{(3)} - \mu_{r,j-1}^{(3)}, \quad r=1, \ldots , n_{3}, j=1, \ldots, n_{2}. $$
Then, by (\ref{q1r})--(\ref{q0r}) it follows that
\begin{align*}
c_{r}^{(3)} = \epsilon_{r} + \sum\limits_{\substack{j=1\\ j \, \text{even}}}^{n_{2}} a_{r,j}^{(3)}, \qquad 
b_{r}^{(3)} = -\epsilon_{r} x_{r,0}^{(3)} - \sum\limits_{\substack{j=1\\ j \, \text{even}}}^{n_{2}} a_{r,j}^{(3)} \, x_{j,0}^{(2)}.
\end{align*}
With this parameter choice we indeed obtain the desired breakpoints $x_{r,j}^{(3)}$.
For $j=0$ we find $$f_{3,r}(x_{r,0}^{(3)}) = \mu_{r,0}^{(3)} x_{r,0}^{(3)} + \eta_{r,0}^{(3)} = q_{1,r}^{(3)}x_{r,0}^{(3)} + q_{0,r}^{(3)} =0. $$
Further, with the recursions 
$$ \mu_{r,j}^{(3)} (x_{j,0}^{(2)} - x_{r,j}^{(3)}) = \mu_{r,j-1}^{(3)} (x_{j,0}^{(2)}- x_{r,j-1}^{(3)}), 
\quad \eta_{r,j}^{(3)} = \eta_{r,j-1}^{(3)} - a_{r,j}^{(3)} x_{j,0}^{(2)}$$
we inductively obtain from $\mu_{r,j-1}^{(3)}x_{r,j-1}^{(3)} + \eta_{r,j-1}^{(3)}=0$ that 
\begin{align*}
\mu_{r,j}^{(3)} x_{r,j}^{(3)} + \eta_{r,j}^{(3)} &= \mu_{r,j}^{(3)}x_{r,j}^{(3)} +\eta_{r,j-1}^{(3)} - a_{r,j}^{(3)} x_{j,0}^{(2)}= \mu_{r,j}^{(3)}x_{r,j}^{(3)}-\mu_{r,j-1}^{(3)} x_{r,j-1}^{(3)}- a_{r,j}^{(3)} x_{j,0}^{(2)}\\
&=\mu_{r,j}^{(3)}x_{r,j}^{(3)}-\mu_{r,j-1}^{(3)}x_{r,j-1}^{(3)}-(\mu_{r,j}^{(3)}- \mu_{r,j-1}^{(3)}) x_{j,0}^{(2)}=0.
\end{align*}
In particular, the parameters $a_{r,j}^{(3)}$ are explicitly given by
\begin{equation}\label{arj3}
a_{r,j}^{(3)} = \textstyle \epsilon_{r} \left( \frac{x_{r,j}^{(3)} - x_{r,j-1}^{(3)}}{x_{j,0}^{(2)} - x_{r,j}^{(3)}} \right) \, \prod\limits_{s=0}^{j-1} \left( \frac{x_{s,0}^{(2)} - x_{r,s-1}^{(3)}}{x_{s,0}^{(2)} - x_{r,s}^{(3)}} \right),
\end{equation}
and hence $\sign(a_{r,j}^{(3)}) = \epsilon_{r} (-1)^{j}$.
Thus, as in Corollary 3.7, the obtained components $a_{r,j}^{(3)}$ have alternating sign with regard to $j$. 

3. 
Corollary \ref{corLL} implies for 
$ f_{4}(t) = t c^{(4)} + b^{(4)} + \sum\limits_{r=1}^{n_{3}} a_r^{(4)} \sigma(f_{3,r}(t))$ with $f_{3,r}(t)$ in (\ref{f3r}) that 
$$ f_{4}(t) = tq_{1}^{(4)}+ q_{0}^{(4)} + \sum_{\ell=1}^{n_{1}n_{2}+n_{1}+n_{2}} \beta_{\ell}^{(4)} \, \sigma(t - \xi_{\ell}^{(3)}) + \sum_{r=1}^{n_{3}} \sum_{\nu=0}^{n_{1}n_{2}+n_{1}+n_{2}} \beta_{r,\nu}^{(4)} \,  \sigma(t-\xi_{r,\nu}^{(4)}), $$
where  $\xi_{(n_{2}+1)\ell}^{(3)} = x_{\ell}^{(1)}, \, \ell=1, \ldots , n_{1}$,  and $\xi_{(n_{2}+1)\nu+j}^{(3)} = x_{j,\nu}^{(2)}$, $j=1, \ldots, n_{2}, \, \nu=0, \ldots , n_{1}$ are the breakpoints of $f_{3,r}$ constructed in the preceding two layers.
Further, the set $\{\xi_{r,\nu}^{(4)}, \, r=1, \ldots , n_{3}, \, \nu=0, \ldots , n_{1}n_{2}+n_{1}+n_{2} \}$ contains the (sub)set $\{x_{r,j}^{(3)}, \, r=1, \ldots , n_{3}, \, j=0, \ldots , n_{2}\}$ by construction. 

4.
We finally check, whether all desired  breakpoints are active in $f_{4}$.
We obtain for $\nu=0, \ldots , n_{2}$, $r=1, \ldots , n_{3}$  that the breakpoints $\xi_{r,\nu}^{(4)} = x_{r,j}^{(3)} \in (x_{j,0}^{(2)}, x_{j+1,0}^{(2)})$ correspond to the  coefficients 
\begin{align*} 
\beta_{r,\nu}^{(4)} &= a_{r}^{(4)} |\mu_{r,\nu}^{(3)}| \neq 0,
\end{align*}
where $({\mathbf A}^{(4)})^{T} = (a_{r}^{(4)})_{r=1}^{n_{3}} \in {\mathbb R}^{n_{3}}$ is chosen  such that $a_{r}^{(4)} \neq 0$.
Next, we show that  we can always choose $(\epsilon_{r})_{r=1}^{n_{3}}$ and ${\mathbf A}^{(4)}$ such that the coefficients $\beta_{(n_{2}+1)\ell}^{(4)}$ corresponding to $x_{\ell}^{(1)}$ and $\beta_{\nu(n_{2}+1)+j}^{(4)}$ corresponding to $x_{j,\nu}^{(2)}$ do not vanish. As in Corollary \ref{corLL} we have
\begin{align*}
\beta_{(n_{2}+1)\ell}^{(4)} &= {\mathbf A}^{(4)} {\mathbf S}^{(\ell)} 
:=\textstyle \!\!\!\!\!\!\sum\limits_{\substack{r=1\\f_{3,r}(x_{\ell}^{(1)}) >0}}^{n_{3}} \!\!\!\!\!\! a_{r}^{(4)} \alpha_{r,\ell} + \!\!\!\!\!\! \sum\limits_{\substack{r=1\\f_{3,r}(x_{\ell}^{(1)}) =0}}^{n_{3}} \!\!\!\!\!\! a_{r}^{(4)} \left(\sigma( \partial_{+}f_{3,r}(x_{\ell}^{(1)}) )+ \sigma( -\partial_{-}f_{3,r}(x_{\ell}^{(1)})) \right),\\
\beta_{\nu(n_{2}+1)+j}^{(4)} &= {\mathbf A}^{(4)} {\mathbf S}^{(\nu,j)} 
:= \textstyle \!\!\!\!\!\!\! \sum\limits_{\substack{r=1\\f_{3,r}(x_{j,\nu}^{(2)}) >0}}^{n_{3}} \!\!\!\!\!\!\! a_{r}^{(4)} \alpha_{r, j,\nu} + \!\!\!\!\!\!\! \sum\limits_{\substack{r=1\\f_{3,r}(x_{j,\nu}^{(2)}) =0}}^{n_{3}} \!\!\!\!\!\!\! a_{r}^{(4)} \left( \sigma( \partial_{+}f_{3,r}(x_{j,\nu}^{(2)}) )+ \sigma( -\partial_{-}f_{3,r}(x_{j,\nu}^{(2)})) \right),
\end{align*}
where, as in Section 3, $\partial_{+}f_{3,r}(x)$ and  $\partial_{-}f_{3,r}(x)$ denote the slopes of $f_{3,r}(x)$ on the right and left side of $x$, 
and where the vectors ${\mathbf S}^{(\ell)}$ and ${\mathbf S}^{(\nu,j)}$ of length $n_{3}$ contain the entries $\alpha_{r,\ell}$ (resp. $\alpha_{r, j,\nu}$) 
for positive function values, $\sigma( \partial_{+}f_{3,r}(x_{\ell}^{(1)}) )+ \sigma( -\partial_{-}f_{3,r}(x_{\ell}^{(1)}))$ 
(resp. $\sigma( \partial_{+}f_{3,r}(x_{j,\nu}^{(2)}) )+ \sigma( -\partial_{-}f_{3,r}(x_{j,\nu}^{(2)}))$) for vanishing function values, and zero entries for negative function values.

We obtain from (\ref{f3r})  with our settings that
\begin{align*}
& f_{3,r}(x_{\ell}^{(1)}) = \textstyle \epsilon_{r} (x_{\ell}^{(1)} - x_{r,0}^{(3)}) + \sum\limits_{k=1}^{\ell-1} \alpha_{r,k} (x_{\ell}^{(1)} - x_{k}^{(1)}) + \! \sum\limits_{\nu=0}^{\ell-1} \sum\limits_{j=1}^{n_{2}} \alpha_{r,j,\nu} (x_{\ell}^{(1)} - x_{j,\nu}^{(2)}) \\
&= \textstyle \epsilon_{r} \left(\! |x_{\ell}^{(1)} - x_{r,0}^{(3)}| + \!\sum\limits_{k=1}^{\ell-1} (-1)^{k+1}|\alpha_{r,k}| |x_{\ell}^{(1)} - x_{k}^{(1)}| + \!\sum\limits_{\nu=0}^{\ell-1} \sum\limits_{j=1}^{n_{2}} (-1)^{j} |\alpha_{r,j,\nu}| |x_{\ell}^{(1)} - x_{j,\nu}^{(2)}| \!\right)
\end{align*}
such that a switch of $\epsilon_{r}$ from $1$ to $-1$ changes the sign of the function value $f_{3,r}(x_{\ell}^{(1)})$.  
Moreover, for function values $f_{3,r}(x_{\ell}^{(1)})=0$ we have that 
 $\sigma( \partial_{+}f_{3,r}(x_{\ell}^{(1)}) )+ \sigma( -\partial_{-}f_{3,r}(x_{\ell}^{(1)}))$ or $\sigma( -\partial_{+}f_{3,r}(x_{\ell}^{(1)}) )+ \sigma( \partial_{-}f_{3,r}(x_{\ell}^{(1)}))$ is positive.
Similarly, also  $\sign (f_{3,r}(x_{j,\nu}^{(2)}))$ is switched by changing $\sign (\epsilon_{r})$.
%Moreover, $\sign( a_{r}^{(4)} \alpha_{r,\ell}) = \sign((-1)^{\ell+1})$ and $\sign( a_{r}^{(4)} \alpha_{r, j,\nu}) = \sign((-1)^{j})$ do not depend on $r$.
%%We assume that the function values $f_{3,r}(x_{\ell}^{(1)})$ and $f_{3,r}(x_{j,\nu}^{(2)}) )$ do not vanish, which is almost sure. 
%Then 
%\begin{align*}
%\beta_{(n_{2}+1)\ell}^{(4)} = (-1)^{\ell+1}\sum\limits_{\substack{r=1\\f_{3,r}(x_{\ell}^{(1)}) >0}}^{n_{3}} |a_{r}^{(4)} \alpha_{r,\ell}|, \qquad 
%\beta_{\nu(n_{2}+1)+j}^{(4)} = (-1)^{j} \sum\limits_{\substack{r=1\\f_{3,r}(x_{j,\nu}^{(2)}) >0}}^{n_{3}} |a_{r}^{(4)} \alpha_{r, j,\nu}| ,
%\end{align*}

It suffices now to show that the vectors ${\mathbf S}^{(\ell)}$ and ${\mathbf S}^{(\nu,j)}$ are all nonzero vectors, then we can always find a vector ${\mathbf A}^{(4)}$ that is not orthogonal to any of these vectors, i.e., such that all coefficients 
$\beta_{(n_{2}+1)\ell}^{(4)}={\mathbf A}^{(4)} {\mathbf S}^{(\ell)}$ and $\beta_{\nu(n_{2}+1)+j}^{(4)}={\mathbf A}^{(4)} {\mathbf S}^{(\nu,j)}$ do not vanish.
%we always have at least one  function value $f_{3,r}(x_{\ell}^{(1)}) \ge 0$ in the first sum and $f_{3,r}(x_{j,\nu}^{(2)}) \ge 0$ in the second sum, such that the sums above are not empty. 
We will ensure that property by choosing the vector ${\mathbf \epsilon} = (\epsilon_{1}, \ldots , \epsilon_{n_{3}}) \in \{ -1,1\}^{n_{3}} $ properly. 
We proceed as follows to fix $\epsilon_{r}$. We choose $\epsilon_{1}$ such that  at least half of the $(n_{1}+1)(n_{2}+1)-1$ function values $f_{3,1}(\xi_{\ell}^{(3)})$ are positive (or zero  with corresponding positive entry in ${\mathbf S}^{(\ell)}$ resp. ${\mathbf S}^{(\nu,j)}$). Next we choose $\epsilon_{2}$ such that at least half of the remaining function values $f_{3,2}(\xi_{\ell}^{(3)})$ are positive (or zero  with corresponding positive entry in ${\mathbf S}^{(\ell)}$ resp. ${\mathbf S}^{(\nu,j)}$), where we had $f_{3,1}(\xi_{\ell}^{(3)}) \le 0$. We repeat this procedure to get for any $\xi_{\ell}^{(3)}$ at least one positive (or zero) value $f_{3,r}(\xi_{\ell}^{(3)})$ and can stop in the worst case after $\lceil\log_{2}(n_{1}n_{2}+ n_{1}+n_{2})\rceil$ steps. 
 
The constructed function $f_{4}$ therefore possesses the wanted $n_{1}n_{2} + n_{2}n_{3}+ n_{1}+ n_{2}+ n_{3}$ breakpoints.
\end{proof}

\begin{remark}
1. The procedure given in the proof of Theorem \ref{theo4.2} is not the only possible method to construct a function $f_{4} \in {\mathcal Y}_{n_{1},n_{2},n_{3}}$. For example, instead of inserting the third-layer breakpoints in the interval $(-\infty,x_{1}^{(1)})$, one could take a different interval $(x_{\nu}^{(1)}, x_{\nu+1}^{(1)})$ to insert these points.

2. The assumption $n_{3} \ge  \log_{2}(n_{1}+(n_{1}+1)n_{2})$ is a technical assumption needed in the proof that can possibly be relaxed, see the Example \ref{ex3} below.

3. The main idea of the proof of Theorem \ref{theo4.2} can be generalized to construct functions $f_{L} \in {\mathcal Y}_{n_{1},n_{2}, \ldots, n_{L-1}}$ with $n_{L-1}n_{L-2}+ \ldots + n_{2}n_{1}+n_{1}+n_{2}+ \ldots +n_{L-1}$ prescribed breakpoints. The most difficult part in that proof is then to show that all breakpoints of the preceding layers stay to be active in the construction. In \cite{Phuong20}, it has just been assumed that for all $t\in {\mathbb R}$ and each layer at least one unit is active. The corresponding \textnormal{ReLU} network is then called \textit{transparent}. In the proof of Theorem \ref{theo4.2}, this would mean that the vectors ${\mathbf S}^{(\ell)}$ and ${\mathbf S}^{(\nu,j)}$ are nonzero.
\end{remark}

\begin{example}\label{ex3}
Let $n_{1}=n_{2}=n_{3}=2$. We construct a function $f_{4} \in {\mathcal Y}_{n_{1},n_{2},n_{3}}$ with the $n_{1}n_{2}+ n_{2}n_{3}+ n_{1}+ n_{2}+ n_{3} = 14$ pre-determined breakpoints $x_{k}=k$, $k=1, \ldots, 14$. 
According to Theorem \ref{theo4.2} we set 
\begin{align*}
x_1^{(1)} &=9, \quad x_2^{(1)}=12,\\
x_{1,0}^{(2)}&=3,\quad x_{2,0}^{(2)}=6,\quad x_{1,1}^{(2)}=10,\quad x_{2,1}^{(2)}=11,\quad x_{1,2}^{(2)}=13,\quad x_{2,2}^{(2)}=14,\\
x_{1,0}^{(3)}&=1,\quad x_{2,0}^{(3)}=2,\quad x_{1,1}^{(3)}=4,\quad x_{2,1}^{(3)}=5,\quad x_{1,2}^{(3)}=7,\quad x_{2,2}^{(3)}=8,
\end{align*} 
see Figure \ref{figknots}.
To construct $f_{3,1}(t)$ and $f_{3,2}(t)$ we set according to Theorem \ref{theo4.2}
$$ {\mathbf A}^{(1)} = \begin{pmatrix} 1 \\ 1 \end{pmatrix}, \; {\mathbf b}^{(1)} = \begin{pmatrix} -9 \\ -12 \end{pmatrix}, \; {\mathbf c}^{(2)} = \begin{pmatrix} 1 \\ -1 \end{pmatrix}, \; {\mathbf b}^{(2)} = \begin{pmatrix} -3 \\ 6 \end{pmatrix}, \; {\mathbf A}^{(2)} = \begin{pmatrix} -7 & 18 \\ \frac{5}{2} & -\frac{9}{4} \end{pmatrix},
$$
and obtain for $r=1,2$,
\begin{align*}
f_{3,r}(t) = t c_{r}^{(3)} + b_{r}^{(3)} + ({a}_{r,1}^{(3)}, {a}_{r,2}^{(3)}) \begin{pmatrix}\sigma(-7\sigma(t-9)+18\sigma(t-12)+t-3) \\
\sigma(\frac{5}{2}\sigma(t-9)-\frac{9}{4}\sigma(t-12)-t+6) \end{pmatrix}.
\end{align*}
Next, we choose $\epsilon_{1}=1$ and $\epsilon_{2}=-1$, i.e.,
\begin{align*}
{\mathbf q}_{1}^{(3)} = \begin{pmatrix} 1 \\ -1 \end{pmatrix},\; {\mathbf q}_{0}^{(3)} = \begin{pmatrix} -1 \\ 2 \end{pmatrix}, \, {\mathbf A}^{(3)} = \begin{pmatrix} -3 & 6 \\ \frac{3}{2} & -\frac{3}{4} \end{pmatrix}
\end{align*}
which leads to ${\mathbf c}^{(3)} = (7, -\frac{7}{4})^T$ and ${\mathbf b}^{(3)} = (-37, \frac{13}{2})^T$.
With these settings, we find
\begin{align*}
f_{3,1}(t) &= \textstyle 7t -37 + (-3, 6) \begin{pmatrix}\sigma(-7\sigma(t-9)+18\sigma(t-12)+t-3) \\
\sigma(\frac{5}{2}\sigma(t-9)-\frac{9}{4}\sigma(t-12)-t+6) \end{pmatrix}\\
&=  \textstyle t-1 -3\sigma(t-3) +6\sigma(t-6) + 21\sigma(t-9) -18 \sigma(t-10) +9\sigma(t-11) \\
& \textstyle \quad -\frac{27}{2} \sigma(t-12) - 36 \sigma(t-13) + \frac{9}{2} \sigma(t - 14)
\end{align*}
with function values $f_{3,1}(9)=8$, $f_{3,1}(12)=56$ at first level breakpoints, and $f_{3,1}(3)=2$, $f_{3,1}(6)=-4$, $f_{3,1}(10)=33$, $f_{3,1}(11)=40$,
$f_{3,1}(13)=58.5$, $f_{3,1}(14)=25$ at second level breakpoints,
and 
\begin{align*}
f_{3,2}(t) &= \textstyle -\frac{7}{4}t +\frac{13}{2} + (\frac{3}{2}, -\frac{3}{4}) \begin{pmatrix}\sigma(-7\sigma(t-9)+18\sigma(t-12)+t-3) \\
\sigma(\frac{5}{2}\sigma(t-9)-\frac{9}{4}\sigma(t-12)-t+6) \end{pmatrix} \\
&= \textstyle -t+2  +\frac{3}{2}\sigma(t-3)  - \frac{3}{4}\sigma(t-6) - \frac{21}{2}\sigma(t-9) +9 \sigma(t-10)  - \frac{9}{8}\sigma(t-11) \\
& \textstyle \quad +\frac{27}{16} \sigma(t-12) +18 \sigma(t-13) - \frac{9}{16} \sigma(t - 14)
\end{align*}
with function values $f_{3,2}(9)=-0.25$, $f_{3,2}(12)=-15.625$ at first level breakpoints, and $f_{3,2}(3)=-1$, $f_{3,2}(6)=0.5$, $f_{3,2}(10)=-11$, $f_{3,2}(11)=-12.75$,
$f_{3,2}(13)=-16.8125$, $f_{3,2}(14)=0$ at second level breakpoints.
Therefore, the choice of $\epsilon_{1}=1$ and $\epsilon_{2}=-1$ leads to non-vanishing coefficients in the spline representation of $f_{4}$. With $a_{1}^{(4)} = -1$ and $a_{2}^{(4)} = 1$, and $c_{1}^{(4)}=b_{1}^{(4)}=0$ we obtain from Lemma \ref{lemsigma}
\begin{align*}
f_4(t) & =-\sigma(f_{3,1}(t))+\sigma(f_{3,2}(t))\\
&= (-t+2) + \sum_{j=1}^{14} \alpha_{j} \, \sigma(t-j) + \alpha_{15} \sigma(t - x_{15})
 \end{align*}
with the wanted breakpoints $x_{k}=k$, $k=1, \ldots , 14$, one additional breakpoint $x_{15}=431/29\approx 14.862$ and corresponding coefficients
$${\balpha} = (\alpha_{j})_{j=1}^{15}= (-1,1,3,-2,0.5,-0.75,-4,0.25,-21,18, -9,13.5,36,11.75,-29)^{T}. $$
The illustration of $f_{4}$ is given in Figure \ref{fig3}.
\begin{figure}[h]
\begin{center}
	\includegraphics[scale=0.6]{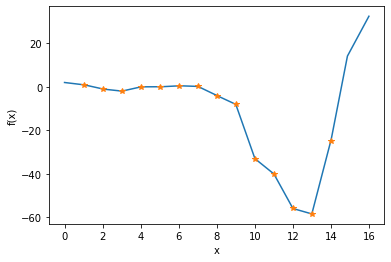}
\end{center}
\caption{Illustration of the function $f_{4}$ obtained in Example \ref{ex3}.}
\label{fig3}
\end{figure}
\end{example}

\subsection{Parameter redundancies in the ReLU model with more hidden layers}

As seen in (\ref{DNN}) and (\ref{LL}) the ReLU model ${\mathcal Y}_{n_{1}, n_{2}, \ldots , n_{L-1}}$ with $L-1$ hidden layers and $n_{0}=n_{L}=1$ possesses $2n_{1}+ (n_{1}+2)n_{2}+ \ldots +(n_{L-1}+2)n_{L}$ parameters.
In Theorem \ref{theored1}, we had shown that for $L=3$, the model can be simplified by fixing ${\mathbf A}^{(1)} = {\mathbf 1}$ and  by restricting ${\mathbf c}^{(2)}$ to a sign vector.
The redundancy, which is due to the positive scaling property $\sigma({\mathbf D} {\mathbf y}) = {\mathbf D} \sigma({\mathbf y})$, obviously occurs at each hidden layer, and we can use this fact in order to normalize ${\mathbf c}^{(\ell)} \in {\mathbb R}^{n_{\ell}}$ for $\ell=2, \ldots, L-1$ and reduce it to a sign vector. 
Summarizing, we obtain 

\begin{corollary}
The model ${\mathcal Y}_{n_{1},n_{2}, \ldots , n_{L-1}}$ in $(\ref{DNN})$ can be equivalently determined by
\begin{align*}
{\mathbf F}_{1}(t) &= t {\mathbf 1} + {\mathbf b}^{(1)} \\
{\mathbf F}_{\ell}(t) &= {\mathbf A}^{(\ell)} \sigma ({\mathbf F}_{\ell-1}(t) ) + (t \, \sign ({\mathbf c}^{(\ell)}) + {\mathbf b}^{(\ell)}), \quad \ell=2, \ldots , L-1\\
f_{L} (t) &= {\mathbf A}^{(L)} \sigma({\mathbf F}_{L-1}(t)) + (t c^{(L)} + b^{(L)}),
\end{align*}
depending on $n_{1}+ (n_{1}+1)n_{2} + (n_{2} + 1)n_{3}+ \ldots  + (n_{L-2}+ 1)n_{L-1} + n_{L-1}+ 2$ real parameters and $n_{2}+n_{3}+ \ldots +n_{L-1}$ sign parameters. 
\end{corollary}

\begin{proof}
The proof follows analogously as for Theorem \ref{theored1} using an induction argument.
\end{proof} 

In particular, ${\mathcal Y}_{n_{1},n_{2},n_{3}}$ depends on at most $n_{1}+ (n_{1}+1)n_{2} + (n_{2} + 1)n_{3} + n_{3}+2$ real parameters and $n_{2}+ n_{3}$ sign parameters. 
In Theorem \ref{theo4.2} we have been able to show that we can always construct $f_{4} \in {\mathcal Y}_{n_{1},n_{2},n_{3}}$ with $n_{1}n_{2}+ n_{2}n_{3}+ n_{1}+ n_{2}+ n_{3}$ prescribed breakpoints. Moreover, in this procedure, the sign vectors ${\mathbf c}^{(2)}$ and ${\mathbf c}^{(3)}$ as well as ${\mathbf A}^{(4)} \in {\mathbb R}^{1 \times n_{3}}$ can be chosen independently of the breakpoints, while some of these choices may lead to functions, where not all prescribed breakpoints are active. We therefore obtain similarly as in Theorem \ref{theored1}

\begin{theorem}\label{theored2}
Any function $f_{4} \in {\mathcal Y}_{n_{1},n_{2},n_{3}}$  can be represented as a function of the form 
\begin{align}\nonumber f_{4}(t) &=  {c}^{(4)} t + {b}^{(4)} \!+ {\mathbf A}^{(4)}\sigma( {\mathbf A}^{(3)} \sigma({\mathbf A}^{(2)} \sigma(t {\mathbf 1} + {\mathbf b}^{(1)}) \!+ t \sign{\mathbf c}^{(2)}  + {\mathbf b}^{(2)}) \!+  t \sign{\mathbf c}^{(3)} \! + {\mathbf b}^{(3)})\\
\label{twol3}
\end{align}
with a sign vectors ${\mathbf c}^{(2)} \in \{-1,0,1\}^{n_{2}}$, ${\mathbf c}^{(3)} \in \{-1,0,1\}^{n_{3}}$ and with $n_{1}+ (n_{1}+1)n_{2} + (n_{2} + 1)n_{3} + n_{3}+2$ real parameters in 
${c}^{(4)}$, ${b}^{(4)}$, ${\mathbf A}^{(\ell)} \in {\mathbb R}^{n_{\ell} \times n_{\ell-1}}$,  ${\mathbf b}^{(\ell)} \in {\mathbb R}^{n_\ell}$, $\ell=1, \ldots , 4$.
Moreover, the parameters  are independent, i.e., any restriction of ${\mathcal Y}_{n_{1},n_{2},n_{3}}$ to a model $\tilde{\mathcal Y}_{n_{1},n_{2},n_{3}}$, where one or more of these parameters are fixed to be zero, leads to 
$$ \tilde{\mathcal Y}_{n_{1},n_{2},n_{3}} \subsetneq {\mathcal Y}_{n_{1},n_{2},n_{3}}. $$
\end{theorem}

%%%%%%%%%%%%%%%%%%%%%%%%%%%%%%%%%%%%%%%%%%%%%%%%%%%%%%%%%%%%%%%%%%%%%%%%%%%%%%%%%%%%%%%%%%%%%%%%%%%%%%%%
\section*{Acknowledgement}
The authors gratefully acknowledge support by the German Research Foundation in the framework of the RTG 2088. The first author acknowledges support by the EU MSCA-RISE-2020 project EXPOWER.  
%%%%%%%%%%%%%%%%%%%%%%%%%%%%%%%%%%%%%%%%%%%%%%%%%%%%%%%%%%%%
\small
%\bibliographystyle{plain}
%\bibliography{bibliography.bib}

\end{document}